\documentclass[12pt,eqno]{article}
\usepackage{}
\date{}
\usepackage{amsfonts}
\usepackage{bbm}
\usepackage{bm}
\usepackage{amsmath}
\usepackage{mathrsfs}
\usepackage{amssymb}
\usepackage{amsmath,color}
\usepackage{amsthm}
\usepackage{amstext}
\usepackage{amsopn}
\usepackage{texdraw}
\usepackage{graphicx}
\usepackage{multirow}
\usepackage{lscape}
\usepackage{float}
\usepackage[palatino,gill,courier]{altfont}

\usepackage{subfigure}
\numberwithin{equation}{section}
\newtheorem{theorem}{Theorem}[section] 

\newtheorem{lemma}[theorem]{Lemma}

\newtheorem{remark}[theorem]{Remark}

\newcommand{\norm}[1]{\left\| #1 \right\|}
\newcommand{\inner}[2]{\left\langle #1, #2 \right\rangle}
\newcommand{\R}{\mathbb{R}}
\newcommand{\D}{\mathcal{D}}
\newcommand{\Hil}{\mathcal{H}}

\newcommand{\kapp}{\kappa}
\newcommand{\dom}[1]{\mathcal{D}\left( #1 \right)}

\newcommand{\M}[1]{M_{#1}}

\newcommand{\oone}{o(1)}
\DeclareRobustCommand{\R}{\mathbb{R}}

\begin{document}
	
	\title{Stability of Magnetizable Piezoelectric Beam Systems with Two Fractional Infinite Memory Terms\footnote{This work is supported by NSFC (12671232).}}
	\author{Jun Zhou\footnote{Corresponding author. Email: jzhou@swu.edu.cn} ~~~~~~ Linfeng Duan\footnote{Email: duanlinfeng29@163.com}\\\\
		{\small School of Mathematics and Statistics, Southwest University},\\
		{\small Chongqing 400715, People's Republic of China}}
	\date{}
	\maketitle
	
\begin{abstract}
	In this paper, we investigate a class of magnetizable piezoelectric beam systems with two fractional infinite memory terms, whose fractional orders are denoted by \(\theta_1, \theta_2 \in [0,1]\). By introducing the Dafermos history variable, we construct an augmented state space and reformulate the system as an abstract evolution equation. The well-posedness of the system is established via semigroup theory. Through a frequency-domain analysis, we characterize the asymptotic behavior of the solutions. We show that the system is exponentially stable if \(\theta_1 = \theta_2 = 1\), and is polynomially stable in all other cases, with the explicit decay rate \(t^{-\frac{1}{2-2\theta_0}}\), where \(\theta_0 = \min\{\theta_1, \theta_2\}\). Furthermore, the optimality of the obtained decay rate is verified.

	\textbf{Keywords}: magnetizable piezoelectric system; fractional infinite memory; frequency-domain analysis; polynomial stability; exponential stability
\end{abstract}

\section{Introduction}
Let $H$ be a complex Hilbert space with inner product $\langle\cdot,\cdot\rangle$, which generates the norm $\|\cdot\|$ on $H$. Let $A:\D(A)\to H$ be a strictly positive self-adjoint linear operator on $H$ with compact resolvent, where $\D(A)$ denotes the domain of $A$. The spectrum of $A$ consists solely of isolated eigenvalues $\{\xi_n\}_{n\in\mathbb{N}}$ satisfying
\[
0 < \xi_1 < \xi_2 < \cdots < \xi_n < \xi_{n+1} < \cdots,\qquad \lim_{n\to\infty}\xi_n = \infty.
\]
For each $n\in\mathbb{N}$, let $e_n\in \D(A)$ be the normalized eigenvector corresponding to $\xi_n$, so that $\|e_n\|=1$ and $A e_n = \xi_n e_n$. The sequence $\{e_n\}_{n=1}^\infty$ forms an orthonormal basis of $H$.

We set $\D(A^0)\triangleq H$. For any $\alpha\in(0,1]$, define
\begin{equation}\label{DAalpha}
	\D(A^{\frac{\alpha}{2}})\triangleq\left\{\phi\in H\,\bigg|\,\|A^{\frac{\alpha}{2}}\phi\|<\infty\right\},
\end{equation}
where the fractional power $A^{\frac{\alpha}{2}}$ acts on every $\phi\in H$ via the series representation
\[
A^{\frac{\alpha}{2}}\phi=\sum_{n=1}^\infty\xi_n^{\frac{\alpha}{2}}\langle\phi,e_n\rangle e_n.
\]
The associated norm identity reads
\[
\|A^{\frac{\alpha}{2}}\phi\|^2=\sum_{n=1}^\infty\xi_n^\alpha\bigl|\langle\phi,e_n\rangle\bigr|^2.
\]
The space $\D(A^{\frac{\alpha}{2}})$ is a Hilbert space when equipped with the inner product
\begin{equation}\label{innerDAalpha}
	\langle\phi_1,\phi_2\rangle_{\D(A^{\frac{\alpha}{2}})}\triangleq\left\langle A^{\frac{\alpha}{2}}\phi_1,A^{\frac{\alpha}{2}}\phi_2\right\rangle,\qquad
	\forall\,\phi_1,\phi_2\in \D(A^{\frac{\alpha}{2}}).
\end{equation}

In this paper, we consider the following abstract magnetizable piezoelectric system with two fractional infinite memory terms
\begin{equation}\label{eq:original}
	\begin{cases}
		\varrho v_{tt} = -\alpha Av(t) + \gamma \beta Ap(t) + \displaystyle\smallint_{0}^{\infty}g_{1}(s)A^{\theta_{1}}v(t - s)\,ds,&t>0,\\[1.2em]
		\mu p_{tt} = -\beta Ap(t) + \gamma \beta Av(t) + \displaystyle\smallint_{0}^{\infty}g_{2}(s)A^{\theta_{2}}p(t - s)\,ds,&t>0,\\[1.2em]
		v(-t) = h_1(t), ~p(-t)=h_2(t),& t>0, \\
		v(0)=v_0\triangleq\lim_{t\to0^+}h_1(t),\ v_t(0)=v_1\triangleq-\lim_{t\to0^+}\frac{d}{dt}h_1(t),\\
		p(0)=p_0\triangleq\lim_{t\to0^+}h_2(t),\ p_t(0)=p_1\triangleq-\lim_{t\to0^+}\frac{d}{dt}h_2(t).
	\end{cases}
\end{equation}
We make the following assumptions:
\begin{description}
	\item[(A1)] $g_j\in L^1(\R_+) \cap H^1(\R_+)$ satisfies $0<\kapp_{j} := \smallint_{0}^{\infty} g_{j}(s)\,ds<\infty$ and $g_j(s)>0$ for $s\in \R_+$, \(j=1,2\);
	\item[(A2)] $ \frac{d}{ds}g_j(s) < 0$ for a.e. $s \in \R_+$ and $j=1,2$, and there exist constants $c_i > 0 ,i=0,1,2,3$,  such that $-c_0g_1(s) \leq g_1'(s) \leq -c_1g_1(s), -c_2g_2(s) \leq g_2'(s) \leq -c_3g_2(s)$;
	\item[(A3)] $\theta_{1} \in [0,1]$, $\theta_{2} \in [0,1]$, the constants $\varrho,\mu,\alpha,\beta,\gamma $ are all positive and satisfy
	\begin{equation}\label{asscs}
		\begin{split}
			&\alpha_{1}\triangleq \alpha - \gamma^2\beta>0,~\alpha-\kappa_1 \xi_1^{\theta_1-1}>0,~\beta-\kappa_2 \xi_1^{\theta_2-1}>0,\\
			&\left(\alpha-\kappa_1 \xi_1^{\theta_1-1}\right)\left(\beta-\kappa_2 \xi_1^{\theta_2-1}\right)>\beta^2\gamma^2.
		\end{split}
	\end{equation}
\end{description}

Piezoelectric materials are a representative class of electromechanically coupled smart materials. Owing to their capability of bidirectional energy conversion between mechanical and electrical domains, they have been widely applied in ultrasonic transducers, actuators, sensors and many other fields \cite{brunner2005potential,yeh2008integrated}. Magnetizable piezoelectric composites, fabricated by combining piezoelectric and magnetostrictive phases, realize multi-field coupling among the mechanical, electric and magnetic fields, and thus exhibit significant advantages in intelligent vibration control, energy harvesting and other engineering scenarios.

Classical piezoelectric beam models are based on the electrostatic approximation, which neglects the influence of dynamic magnetic effects. In high-frequency or strong-magnetic-field environments, however, such electromagnetic dynamic effects cannot be ignored. To overcome this limitation, Morris and \"Ozer \cite{morris2013strong} employed Hamilton's variational principle to establish, for the first time, a piezoelectric beam model with full magnetic effects, whose dynamic behavior is governed by the following strongly coupled system:
\begin{equation}\label{eq:initial}
  \begin{cases}
    \varrho v_{tt}(x,t) - \alpha v_{xx}(x,t) + \gamma\beta p_{xx}(x,t) = 0, \hspace{5mm} x \in (0,L),\ t>0,\\[2mm]
    \mu p_{tt}(x,t) - \beta p_{xx}(x,t) + \gamma\beta v_{xx}(x,t) = 0,
    \hspace{5mm} x \in (0,L),\ t>0.
  \end{cases}
\end{equation}
Here \(v(x,t)\) denotes the transverse displacement of the beam, and \(p(x,t)\) denotes the total electric displacement load in the transverse direction at each point. The coefficients \(\varrho, \alpha, \gamma, \mu, \beta\) are all positive constants: \(\varrho\) is the mass density per unit volume, \(\alpha\) the elastic stiffness, \(\gamma\) the piezoelectric coefficient, \(\mu\) the magnetic permeability, and \(\beta\) the dielectric impermeability coefficient.

Within the framework of the fully magnetic piezoelectric beam model, the stability problem has been extensively studied, and the conclusions turn out to depend crucially on the manner in which damping is applied. Morris and \"Ozer systematically analyzed single-boundary feedback control under voltage control. They showed that, when the ratio of the wave speeds is irrational, the system is strongly stable but not exponentially stable; when this ratio is a quotient of two odd integers, the system fails even to be strongly stable; and exponential stabilization is achievable only when the ratio corresponds to a coprime pair consisting of one odd and one even integer \cite{morris2013strong,morris2014modeling}. In \cite{morris2014comparison}, they extended the analysis to current control and revealed that the corresponding control operator is bounded and compact, which rules out exponential stability and leaves the system merely strongly stable. Ramos et al. \cite{ramos2018exponential} introduced internal frictional damping in only one of the two equations and obtained exponential stability via the energy method. Subsequently, in \cite{ramos2019equivalence}, they further proved that applying frictional damping simultaneously to both coupled variables at the same boundary yields exponential stability. These results reveal that the type and location of damping have a fundamental impact on the stability of the system. While these studies laid the foundation for the vibration control of magnetizable piezoelectric beams, they did not take into account the intrinsic memory properties of piezoelectric materials.

In practical engineering, piezoelectric materials generally exhibit memory effects, in the sense that their mechanical and electrical behavior depends not only on the current state but also on the past history. Such memory effects, rooted in the Boltzmann superposition principle, are usually modeled by Volterra integrals over the infinite past, which characterize the dependence of the current state on the history of the system. Dafermos \cite{dafermos1970asymptotic} introduced the history variable transformation method, providing a unified framework for the mathematical analysis of memory-dissipative systems. Based on this method, various structural systems with memory-type dissipation have been extensively studied. For instance, Astudillo and Oquendo \cite{astudillo2021stability} introduced a fractional-order operator into the memory term of the Timoshenko beam system and obtained a polynomial decay rate for the solutions.

In recent years, the concept of memory damping has gradually been extended to magnetizable piezoelectric beam systems. According to the location of the memory term, such damping can be classified into internal memory damping acting in the displacement equation and boundary memory damping. For the internal case, Zhang et al.~\cite{zhang2022stability} investigated a multi-dimensional nonlinear magnetizable piezoelectric beam system on a bounded domain, in which a single viscoelastic infinite memory term $\smallint_0^\infty g(s)\Delta v(x,t-s)\,ds$ is imposed only on the displacement equation. By means of semigroup theory and the Banach fixed-point theorem, they established the well-posedness of the system; via frequency-domain analysis, they proved that the corresponding linear coupled system can be indirectly exponentially stabilized by this single memory term, and, notably, the exponential decay rate is independent of the relationships among the wave speeds. They further obtained exponential stability of the nonlinear system for small initial data by the energy method. Subsequently, Zhang et al.~\cite{zhang2025decay} were the first to introduce a fractional-order operator into the memory term of magnetizable piezoelectric beam systems. They considered a strongly coupled hyperbolic system whose memory term takes the form $\smallint_0^\infty g(s)A^\alpha v(t-s)\,ds$ with a fractional order $\alpha\in[0,1)$, where $A$ is a positive self-adjoint operator (the beam model being recovered by $A=-\partial_x^2$). By the frequency-domain method, they proved that the system is polynomially stable with the explicit decay rate $t^{-\frac{1}{2-2\alpha}}$, which depends only on $\alpha$. In the case of exponentially decreasing kernels, they derived the asymptotic expressions of the spectrum of the system operator and, on this basis, verified the optimality of the obtained decay rate through spectral analysis. For boundary memory damping, Poblete et al.~\cite{poblete2023polynomial} studied boundary dissipation with fractional-order derivatives and obtained polynomial decay behavior with rate $t^{-\frac{1}{2-2\alpha}}$ through frequency-domain methods. Feng and {\"O}zer \cite{feng2023stability} analyzed a piezoelectric beam model with long-range memory effects on the boundary, deriving a general decay rate determined by the memory kernel function; when the kernel decays exponentially, they further established exponential stability.

In addition, significant progress has been made on piezoelectric beam systems with time delays and thermal effects. By employing Kato's variable norm technique together with the Lyapunov functional method, Liao \cite{liao2025exponential} proved the exponential stability of magnetizable piezoelectric beams with degenerate memory and time-varying delays. Via the energy method, Ramos et al.\ \cite{ramos2021exponential} established the exponential stability of both fully dynamic and electrostatic piezoelectric beams with distributed time-delay damping feedback; further related results can be found in \cite{peng2019time,peng2019vibration}. Akil \cite{akil2022stability} investigated the stability of magnetizable piezoelectric beam systems coupled with non-Fourier thermal effects, showing that the coupled system with heat conduction governed by the Coleman--Gurtin law is exponentially stable, whereas the one governed by the purely memory-type Gurtin--Pipkin law admits only polynomial decay of order $t^{-1}$. From the viewpoint of infinite-dimensional dynamical systems, Liu et al.\ \cite{liu2025long} studied the long-time behavior of magnetizable piezoelectric systems with historical memory and time-varying delays, and proved the existence of global and exponential attractors.

However, most of the existing studies consider memory damping acting only on the displacement equation. In real physical materials, mechanical deformation and dielectric relaxation processes often exhibit history-dependent memory effects simultaneously. Therefore, a magnetizable piezoelectric beam model in which fractional-order infinite memory damping is applied to both equations, accounting for the mechanical and electrical memory effects simultaneously, better reflects physical reality. Despite the aforementioned studies, systematic theoretical results for such strongly coupled systems with two memory terms are still lacking. Motivated by this observation and based on the single-memory case studied by Zhang et al.\ \cite{zhang2025decay}, in the present paper we introduce fractional-order infinite memory damping into both the displacement equation and the charge equation, thereby constructing the strongly coupled system \eqref{eq:original} with dual memory terms.

We aim to address the following questions: when both equations of the strongly coupled system \eqref{eq:original} contain infinite memory terms, can the system be stabilized by these memory terms? If so, what type of decay can be achieved, and how does the decay rate of the solutions depend on the memory orders \(\theta_1,\theta_2\)?

The main results of this paper are summarized as follows:
\begin{itemize}
	\item System \eqref{eq:original} is well-posed: the operator \(B\) associated with \eqref{eq:original} generates a \(\mathcal{C}_0\) semigroup of contractions \(\{S(t)\}_{t\ge0}\) on the extended state space \(\mathcal{H}\), and hence the system admits a unique mild solution for every initial state \(X_0\in\mathcal{H}\), which becomes a classical solution provided \(X_0\in\mathcal{D}(B)\) (Theorem \ref{thmwellposednerss});
	\item If at least one of the fractional orders is strictly smaller than \(1\), that is, \(\theta_1<1\) or \(\theta_2<1\), then the system is polynomially stable with the explicit decay rate \(t^{-\frac{1}{2-2\theta_0}}\), where \(\theta_0=\min\{\theta_1,\theta_2\}\). In other words, the asymptotic behavior of the system is dominated by the smaller fractional order (Theorem \ref{thmpoly});
	\item When the memory kernel is of exponential type, the polynomial decay rate \(t^{-\frac{1}{2-2\theta_0}}\) is optimal, which is verified by the spectral analysis of the system operator (Theorem \ref{thmopt});
	\item The system is exponentially stable if and only if \(\theta_1=\theta_2=1\) (Theorem \ref{thmex}).
\end{itemize}

\begin{remark}\label{rem:comparison}
	It is worthwhile to compare the results obtained above with the closely related works \cite{zhang2022stability,zhang2025decay}, in which only a single infinite memory term is imposed on the displacement equation of magnetizable piezoelectric beam systems.
	
	(i) In \cite{zhang2022stability}, a multi-dimensional nonlinear magnetizable piezoelectric beam system with a single viscoelastic infinite memory term \(\smallint_0^\infty g(s)\Delta v(x,t-s)\,ds\) acting only on the displacement equation was considered. It was proved that such a single memory term can indirectly exponentially stabilize the whole strongly coupled system, and the decay rate is independent of the relationships among the wave speeds. Theorem \ref{thmex} of the present paper shows that this exponential stabilization remains valid when both equations are equipped with viscoelastic infinite memories (\(\theta_1=\theta_2=1\)).
	
	(ii) In \cite{zhang2025decay}, a fractional operator of order \(\alpha\in[0,1)\) was introduced into the single memory term, and the system was shown to be polynomially stable with the optimal decay rate \(t^{-\frac{1}{2-2\alpha}}\). Theorem \ref{thmpoly} extends this result to the case of two memory terms with possibly different fractional orders: the decay rate keeps the same form \(t^{-\frac{1}{2-2\theta}}\), but is determined solely by the smaller order \(\theta_0=\min\{\theta_1,\theta_2\}\). The proof of optimality follows the spectral analysis strategy developed therein, yet is considerably more involved since two history variables with different fractional orders are coupled.
	
	(iii) The most interesting phenomenon revealed by the dual-memory model is that the long-time behavior of the system is dominated by the ``weakest'' memory term. Indeed, even if one memory term is fully viscoelastic (for instance, \(\theta_1=1\)), the system still decays only polynomially with rate \(t^{-\frac{1}{2-2\theta_2}}\) whenever the order of the other satisfies \(\theta_2<1\). In other words, exponential stability is recovered if and only if both memory terms degenerate to the classical viscoelastic infinite memories (\(\theta_1=\theta_2=1\)). This is quite different from the single-memory case in \cite{zhang2022stability}, where one viscoelastic memory term already guarantees exponential stability.
\end{remark}

The remainder of this paper is organized as follows. In Section 2, we present the functional setting and prove the well-posedness of system \eqref{eq:original} by the \(C_0\)-semigroup theory. In Section 3, we study the stability of the system with dual memory terms, establish the polynomial stability together with its degeneration to the viscoelastic damping case (\(\theta_1=\theta_2=1\)), and verify the optimality of the polynomial decay rate.

	\section{Well-Posedness}\label{sec2}
First, we introduce the state space $\mathcal{H}$ as follows:
\[
\Hil\triangleq\D(A^{\frac{1}{2}})\times H\times \D(A^{\frac{1}{2}})\times H\times M_{1}\times M_{2},
\]
where  $\D(A^{\frac{1}{2}})$ is the Hilbert space given by \eqref{DAalpha}, and for $i=1,2$,
\begin{align}\label{Mi}
M_i\triangleq\left\{\psi:\mathbb{R}_+\mapsto D(A^{\frac{\theta_i}{2}})\bigg|\|\psi\|_{M_i}^2=\smallint_0^\infty g_i(s)\|A^{\frac{\theta_i}{2}}\psi(s)\|^2ds\right\}.
\end{align}
It is obvious that $M_i$, $i=1,2$, are Hilbert spaces with inner produce $\langle\cdot,\cdot\rangle_{M_i}$ being given by
\begin{align}\label{innerMi}
\langle\psi_1,\psi_2\rangle_{M_i}\triangleq\smallint_0^\infty g_i(s)\left\langle A^{\frac{\theta_i}{2}}\psi_1(s),A^{\frac{\theta_i}{2}}\psi_2(s)\right\rangle ds,~~~\forall \psi_1,\psi_2\in M_i.
\end{align}

The state space $\mathcal{H}$ is a Hilbert space with inner product $\langle\cdot,\cdot\rangle_{\mathcal{H}}$ defined by
\begin{align*}
		\bigl\langle (v,u,p,q,\eta,\theta)^{\top},(\tilde{v},\tilde{u},\tilde{p},\tilde{q},\tilde{\eta},\tilde{\theta})^{\top}\bigr\rangle_{\Hil} =& \alpha_{1}\bigl\langle A^{\frac{1}{2}}v,A^{\frac{1}{2}}\tilde{v}\bigr\rangle
		-\kapp_{1}\bigl\langle A^{\frac{\theta_1}{2}}v,A^{\frac{\theta_1}{2}}\tilde{v}\bigr\rangle
		+ \varrho \langle u,\tilde{u}\rangle\\
		&+\beta \bigl\langle \gamma A^{\frac{1}{2}}v - A^{\frac{1}{2}}p,\gamma A^{\frac{1}{2}}\tilde{v} - A^{\frac{1}{2}}\tilde{p}\bigr\rangle \\
		&-\kapp_{2}\bigl\langle A^{\frac{\theta_2}{2}}p,A^{\frac{\theta_2}{2}}\tilde{p}\bigr\rangle
		+ \mu \langle q,\tilde{q}\rangle
		+\langle \eta,\tilde{\eta}\rangle_{M_{1}} + \langle \theta,\tilde{\theta}\rangle_{M_{2}}.
	\end{align*}
Thus, the norm $\|\cdot\|_{\mathcal{H}}$ introduced by the inner product is given by
\[
		\begin{aligned}
			\|(v,u,p,q,\eta,\theta)^{\top}\|_{\Hil}^{2} &= \alpha_{1}\bigl\|A^{\frac{1}{2}}v\bigr\|^{2} - \kapp_{1}\bigl\|A^{\frac{\theta_1}{2}}v\bigr\|^{2}
			+ \varrho \|u\|^{2} + \beta \bigl\|\gamma A^{\frac{1}{2}}v - A^{\frac{1}{2}}p\bigr\|^{2}\\
			&\quad- \kapp_{2}\bigl\|A^{\frac{\theta_2}{2}}p\bigr\|^{2}  +\mu \|q\|^{2} + \|\eta\|_{M_{1}}^{2} + \|\theta\|_{M_{2}}^{2}.
		\end{aligned}
	\]
Next we show the norm of $\mathcal{H}$ defined above is equivalent to the standard norm of $\mathcal{H}$.
	\begin{theorem}\label{thmequinorm}
		There exists two positive constants $\varrho_1$ and $\varrho_2$ independent of $X=(v,u,p,q,\eta,\theta)\in\mathcal{H}$ such that
		$$\varrho_1\|X\|_{\rm stan}\le \|X\|_{\mathcal{H}}\le\varrho_2\|X\|_{\rm stan},$$
		where
		$$\|X\|_{\rm stan}^2\triangleq\|A^{\frac{1}{2}}v\|^2+\|u\|^2+\|A^{\frac{1}{2}}p\|^2+\|q\|^2+\|\eta\|_{\M{1}}^2+\|\theta\|_{\M{2}}^2.$$
	\end{theorem}
\begin{proof}
By \eqref{asscs} in the Assumption (A3), we can choose a positive constant $\varepsilon$ such that
$$\frac{\gamma\beta}{\beta-\kappa_2\xi_1^{\theta_2-1}}<\varepsilon<\frac{\alpha-\kappa_1\xi_1^{\theta_1-1}}{\gamma\beta}.$$
Then
$$\delta_1\triangleq\alpha-\kappa_1\xi_1^{\theta_1-1}-\gamma\beta\varepsilon>0,~~~\delta_2\triangleq\beta-\kappa_2\xi_1^{\theta_2-1}-\frac{\gamma\beta}{\varepsilon}>0.$$
Moreover, since $\alpha\le 1$, we get
		\begin{align*}
			\|A^{\frac{\alpha}{2}}v\|^2=\sum_{n=1}^\infty\xi_n^\alpha|\langle v,e_n\rangle|^2=\sum_{n=1}^\infty\xi_n^{\alpha-1}\xi_n|\langle v,e_n\rangle|^2\le\xi_1^{\alpha-1}\sum_{n=1}^\infty\xi_n|\langle v,e_n\rangle|^2=\xi_1^{\alpha-1}\|A^{\frac{1}{2}}v\|^2.
		\end{align*}
Applying Young's inequality yields
		\begin{align}\label{coer}
			& \alpha_{1}\bigl\|A^{\frac{1}{2}}v\bigr\|^{2} - \kapp_{1}\bigl\|A^{\frac{\theta_1}{2}}v\bigr\|^{2}+ \beta \bigl\|\gamma A^{\frac{1}{2}}v - A^{\frac{1}{2}}p\bigr\|^{2}- \kapp_{2}\bigl\|A^{\frac{\theta_2}{2}}p\bigr\|^{2} \notag\\
			 \ge & \alpha_{1}\bigl\|A^{\frac{1}{2}}v\bigr\|^{2}-\kappa_1\xi_1^{\theta_1-1}\bigl\|A^{\frac{1}{2}}v\bigr\|^{2}\notag\\
&+\beta\left(\bigl\|A^{\frac{1}{2}}p\bigr\|^{2}-2\gamma\bigl\|A^{\frac{1}{2}}v\bigr\|\bigl\|A^{\frac{1}{2}}p\bigr\|+\gamma^2\bigl\|A^{\frac{1}{2}}v\bigr\|^{2}\right)-\kappa_2\xi_1^{\theta_2-1}\bigl\|A^{\frac{1}{2}}p\bigr\|^{2}\notag\\
\ge&\delta_1\bigl\|A^{\frac{1}{2}}v\bigr\|^{2}+\delta_2\bigl\|A^{\frac{1}{2}}p\bigr\|^{2}.
		\end{align}
Therefore, we get
		\begin{align*}
			\|X\|_{\mathcal{H}}^2\ge\delta_1\bigl\|A^{\frac{1}{2}}v\bigr\|^{2}+ \varrho \|u\|^{2}+\delta_2\bigl\|A^{\frac{1}{2}}p\bigr\|^{2}
			 +\mu \|q\|^{2} + \|\eta\|_{M_{1}}^{2} + \|\theta\|_{M_{2}}^{2}			\ge\varrho_1\|X\|_{\rm stan}^2,
		\end{align*}
where $\varrho_1\triangleq \min\{\delta_1,\varrho,\delta_2,\mu, 1\}$.	
		On the other hand, it is obvious that that $\|X\|_{\mathcal{H}}\le\varrho_2\|X\|_{\rm stan}$ for some positive constant $\varrho_2$ independent of $X$. Then we get the conclusion.
\end{proof}

Next, to transform the non-autonomous system \eqref{eq:original} to an autonomous system we define two new variables
$\eta^{t}(s)$ and $\theta^{t}(s)$ as
	\[
	\eta^{t}(s) = v(t) - v(t - s),\qquad \theta^{t}(s) = p(t) - p(t - s).
	\]
Then, the system \eqref{eq:original} may be reformulated as follows:
	\begin{equation}\label{eq:rewritten}
		\begin{cases}
			\varrho v_{tt} = -\alpha Av(t) + \gamma \beta Ap(t) + \kapp_{1}A^{\theta_{1}}v(t) - \displaystyle\smallint_{0}^{\infty}g_{1}(s)A^{\theta_{1}}\eta^{t}(s)\,ds,& t>0,\\[1.5em]
			\mu p_{tt} = -\beta Ap(t) + \gamma \beta Av(t) + \kapp_{2}A^{\theta_{2}}p(t) - \displaystyle\smallint_{0}^{\infty}g_{2}(s)A^{\theta_{2}}\theta^{t}(s)\,ds,& t>0,\\
v(-t) = h_1(t), ~p(-t)=h_2(t),& t>0, \\
   		v(0)=v_0,\ v_t(0)=v_1,\\
   p(0)=p_0,\ p_t(0)=p_1.
		\end{cases}
	\end{equation}
Therefore, the problem \eqref{eq:rewritten} can be reformulated into the following abstract autonomous evolution equation in $\mathcal{H}$:
\begin{equation}\label{modelmain}
  \begin{cases}
		\displaystyle \frac{d X(t)}{dt} = B X(t), & t>0, \\[0.8em]
		X(0) = X_0\triangleq (v_0,v_1,p_0,p_1,\eta^0,\theta^0)^{\top},
	\end{cases}
\end{equation}
where $X(t)\triangleq\big(v(t),u(t),p(t),q(t),\eta^t(s),\theta^t(s)\big)^\top$ with $u(t)\triangleq v_t(t)$ and $q(t)\triangleq p_t(t)$, $\eta^0\triangleq\eta^0(s)=v_0-h_1(s)$, $\theta^0\triangleq\theta^0(s)=p_0-h_2(s)$, and $B:\D(B)\subset \mathcal{H}\mapsto\mathcal{H}$ is a linear operator given by
\begin{equation}\label{defB}
B\begin{pmatrix} v\\ u\\ p\\ q\\ \eta\\ \theta \end{pmatrix}
		\triangleq \begin{pmatrix}
			u\\[0.4em]
			\frac{1}{\varrho}\bigl(-\alpha Av + \gamma\beta Ap + \kapp_{1}A^{\theta_{1}}v - \smallint_{0}^{\infty}g_{1}(s)A^{\theta_{1}}\eta(s)\,ds\bigr)\\[0.4em]
			q\\[0.4em]
			\frac{1}{\mu}\bigl(-\beta Ap + \gamma\beta Av + \kapp_{2}A^{\theta_{2}}p - \smallint_{0}^{\infty}g_{2}(s)A^{\theta_{2}}\theta(s)\,ds\bigr)\\[0.4em]
			u - \eta_{s}\\
			q - \theta_{s}
		\end{pmatrix}
\end{equation}
where $(v,u,p,q,\eta,\theta)^\top\in \D(B)$ and
\begin{align*}
\dom{B} &\triangleq \{X\in \Hil \ | BX\in\Hil\}\\
&=\left\{ X=\begin{pmatrix} v\\ u\\ p\\ q\\ \eta\\ \theta \end{pmatrix} \ \middle| \
	\begin{aligned}
		&u \in \dom{A^{\frac{1}{2}}},\ v \in \dom{A^{\frac{1}{2}}},\ p \in \dom{A^{\frac{1}{2}}},\ q \in \dom{A^{\frac{1}{2}}},\\
 &\eta \in \M{1},\ \eta_s \in \M{1},\ \theta \in \M{2},\ \theta_s \in \M{2}, \eta(0)=\theta(0)=0,\\
		&- \alpha A v + \gamma \beta A p + \kappa_1 A^{\theta_1} v - \smallint_0^\infty g_1(s) A^{\theta_1} \eta^t(s) ds \in H, \\
		&- \beta A p + \gamma \beta A v + \kappa_2 A^{\theta_2} p - \smallint_0^\infty g_2(s) A^{\theta_2} \theta^t(s) ds \in H
	\end{aligned}
	\right\}.
\end{align*}
	
	\begin{theorem}[Well-posedness]\label{thmwellposednerss}
		Assume that the Assumptions (A1)-(A3) hold. Then the operator $B$ generates a \(\mathcal{C}_0\) semigroup of contractions $\{S(t)\}_{t\ge0}$ on the Hilbert space \({\mathcal{H}}\). Then, for any $X_0\in\mathcal{H}$, problem \eqref{modelmain} admits a unique mild solution
		\[
		X(t)\triangleq S(t)X_0\in C([0,\infty);\mathcal{H}).
		\]
		Moreover, if $X_0\in \D(B)$, then
		\[
		X(t)\in C([0,\infty); \D(B)) \cap C^1([0,\infty); \mathcal{H})
		\]
		is the classical solution of problem \eqref{modelmain}.
	\end{theorem}
\begin{proof}By Classical L\"umer-Phillips theorem, to complete the proof, we only need to show $B$ is dissipative and $0\in\rho(B)$, where $\rho(B)$ denotes the resolvent set of $B$. For any \(X = (v,u,p,q,\eta,\theta)^{\top}\in \D(B)\), we have
	\begin{align}
		\mathrm{Re}\bigl\langle BX,X\bigr\rangle_{\Hil}
		&=\mathrm{Re}\Bigl\{ \alpha_{1}\bigl\langle A^{\frac{1}{2}}u,A^{\frac{1}{2}}v\bigr\rangle
		-\kapp_{1}\bigl\langle A^{\frac{\theta_1}{2}}u,A^{\frac{\theta_1}{2}}v\bigr\rangle \notag\\
		&\quad +\Bigl\langle -\alpha Av + \gamma\beta Ap + \kapp_{1}A^{\theta_{1}}v
		- \smallint_{0}^{\infty}g_{1}(s)A^{\theta_{1}}\eta^{t}(s)\,ds,\; u\Bigr\rangle \notag\\
		&\quad +\beta \bigl\langle \gamma A^{\frac{1}{2}}u - A^{\frac{1}{2}}q,\;
		\gamma A^{\frac{1}{2}}v - A^{\frac{1}{2}}p\bigr\rangle
		-\kapp_{2}\bigl\langle A^{\frac{\theta_2}{2}}q,A^{\frac{\theta_2}{2}}p\bigr\rangle \notag\\
		&\quad +\Bigl\langle -\beta Ap + \gamma\beta Av + \kapp_{2}A^{\theta_{2}}p
		- \smallint_{0}^{\infty}g_{2}(s)A^{\theta_{2}}\theta^{t}(s)\,ds,\; q\Bigr\rangle \notag\\
		&\quad +\bigl\langle u - \eta_{s},\,\eta\bigr\rangle_{M_{1}}
		+\bigl\langle q - \theta_{s},\,\theta\bigr\rangle_{M_{2}} \Bigr\}\notag\\
		&= \mathrm{Re}\Bigl\{-\smallint_{0}^{\infty}g_{1}(s)\bigl\langle A^{\frac{\theta_1}{2}}\eta_{s},
		A^{\frac{\theta_1}{2}}\eta\bigr\rangle\,ds
		- \smallint_{0}^{\infty}g_{2}(s)\bigl\langle A^{\frac{\theta_2}{2}}\theta_{s},
		A^{\frac{\theta_2}{2}}\theta\bigr\rangle\,ds\Bigr\}.\notag
	\end{align}
Since $\eta\in M_1$, $g_1(s)\|A^{\frac{\theta_1}{2}}\eta(s)\|^2$ is nonnegative and integrable on $\mathbb{R}_+$, then there exists a positive and increasing sequence $\{s_n\}$ such that $s_n\to\infty$ and $g_1(s_n)\|A^{\frac{\theta_1}{2}}\eta(s_n)\|^2\to0$ as $n\to\infty$. Then by $\eta(0)=0$, we get
	\begin{align}\label{eq:Re1}
	\mathrm{Re}\smallint_{0}^{\infty}g_{1}(s)\bigl\langle A^{\frac{\theta_1}{2}}\eta_{s},A^{\frac{\theta_1}{2}}\eta\bigr\rangle\,ds=&\frac12\lim_{n\to\infty}\smallint_0^{s_n}g_1(s)\frac{d}{ds}\bigl\|A^{\frac{\theta_1}{2}}\eta\bigr\|^{2}\,ds\notag\\
=&\frac12\lim_{n\to\infty}\left( g_1(s_n)\bigl\|A^{\frac{\theta_1}{2}}\eta(s_n)\bigr\|^2-\smallint_0^{s_n}g_{1}'(s)\bigl\|A^{\frac{\theta_1}{2}}\eta\bigr\|^{2}\,ds\right)\notag\\
	= &-\frac{1}{2}\smallint_{0}^{\infty}g_{1}'(s)\bigl\|A^{\frac{\theta_1}{2}}\eta\bigr\|^{2}\,ds,
	\end{align}
	and similarly for \(\theta\). Hence
\begin{equation}\label{disB}
  \mathrm{Re}\bigl\langle BX,X\bigr\rangle_{\Hil}
	= \frac{1}{2}\smallint_{0}^{\infty}g_{1}'(s)\bigl\|A^{\frac{\theta_1}{2}}\eta\bigr\|^{2}\,ds
	+ \frac{1}{2}\smallint_{0}^{\infty}g_{2}'(s)\bigl\|A^{\frac{\theta_2}{2}}\theta\bigr\|^{2}\,ds \leq 0.
\end{equation}
   Therefore, the operator \(B\) is dissipative.
	Moreover, it remains to show that $0\in\rho(B)$.
	For any \(F = (f_{1},f_{2},f_{3},f_{4},f_{5},f_{6})\in \Hil\), we solve \(BX = F\),\(X\in \D(B)\). This yields the system
	\begin{align}
		u = f_{1},\label{eq:Th1.1}\\
		\displaystyle -\alpha Av + \gamma\beta Ap + \kapp_{1}A^{\theta_{1}}v
		- \smallint_{0}^{\infty}g_{1}(s)A^{\theta_{1}}\eta\,ds= \varrho f_{2},\label{eq:Th1.2}\\
		q = f_{3},\label{eq:Th1.3}\\
		\displaystyle -\beta Ap + \gamma\beta Av + \kapp_{2}A^{\theta_{2}}p-\smallint_{0}^{\infty}g_{2}(s)A^{\theta_{2}}\theta\,ds= \mu f_{4},\label{eq:Th1.4}\\
		u - \eta_{s} = f_{5},\label{eq:Th1.5}\\
		q - \theta_{s} = f_{6}.\label{eq:Th1.6}
	\end{align}
From \eqref{eq:Th1.1} and the condition $\eta(0)=0$, one can solve the ODE \eqref{eq:Th1.5} to obtain
    \begin{equation}\label{eq:Th1.7}
	\eta(s) = \smallint_{0}^{s} f_{1}\,d\tau - \smallint_{0}^{s} f_{5}\,d\tau = s f_{1} - \smallint_{0}^{s} f_{5}\,d\tau.
    \end{equation}

We now verify that $\eta\in M_{1}$ and $\eta_{s}\in M_{1}$. Since $f_1\in \mathcal{D}(A^{\frac{1}{2}})$ and the embedding $\mathcal{D}(A^{\frac{1}{2}})\hookrightarrow \mathcal{D}(A^{\frac{\theta_1}{2}})$ is continuous, we readily have $f_1\in M_1$. Moreover, owing to $f_5\in M_1$ and the relation $\eta_{s}=f_{1}-f_{5}$, it follows that $\eta_s\in M_1$ as well.
We further estimate the $M_1$-norm of $\eta$. One has
$$\bigl\|A^{\frac{\theta_1}{2}}\eta\|^2\le 2s^2\bigl\|A^{\frac{\theta_1}{2}}f_1\|^2+2s\smallint_0^s\bigl\|A^{\frac{\theta_1}{2}}f_5\|^2d\tau.$$
Since the kernel function $g_1(s)$ is non-increasing, we deduce
\begin{align}\label{youxian}
\smallint_0^{s_n} g_1(s)\bigl\|A^{\frac{\theta_1}{2}}\eta\|^2\,ds\le2s_n^2\kappa_1\bigl\|A^{\frac{\theta_1}{2}}f_1\|^2+2s_n^2\|f_5\|_{M_1}^2<\infty,
\end{align}
where $\{s_n\}_{n\ge1}$ is the sequence introduced in the proof of \eqref{eq:Re1}. By virtue of Assumption (A2) and estimate \eqref{eq:Re1}, we obtain
	\begin{align*}
		\smallint_{0}^{s_n} g_{1}(s)\bigl\|A^{\frac{\theta_1}{2}}\eta\bigr\|^{2}\,ds &\le -\frac{1}{c_{1}}\smallint_{0}^{s_n} g_{1}'(s)\bigl\|A^{\frac{\theta_1}{2}}\eta\bigr\|^{2}\,ds \\
		&= \frac{2}{c_{1}}\,\mathrm{Re}\smallint_{0}^{s_n} g_{1}(s)\bigl\langle A^{\frac{\theta_1}{2}}\eta_{s},A^{\frac{\theta_1}{2}}\eta\bigr\rangle\,ds -g_1(s_n)\bigl\|A^{\frac{\theta_1}{2}}\eta(s_n)\bigr\|^{2}\\
		&\le \frac{2}{c_{1}} \|\eta_{s}\|_{M_{1}} \left(\smallint_{0}^{s_n} g_{1}(s)\bigl\|A^{\frac{\theta_1}{2}}\eta\bigr\|^{2}\,ds\right)^{\frac{1}{2}}-g_1(s_n)\bigl\|A^{\frac{\theta_1}{2}}\eta(s_n)\bigr\|^{2}.
	\end{align*}
Define
\begin{align*}
\mathfrak{A}\triangleq\left(\smallint_{0}^{s_n} g_{1}(s)\bigl\|A^{\frac{\theta_1}{2}}\eta\bigr\|^{2}\,ds\right)^{\frac{1}{2}},~ \mathfrak{B}\triangleq \frac{2}{c_{1}} \|\eta_{s}\|_{M_{1}},~\mathfrak{C}\triangleq g_1(s_n)\bigl\|A^{\frac{\theta_1}{2}}\eta(s_n)\bigr\|^{2}.
\end{align*}
It is clear that $\mathfrak{B}<\infty$. In view of the construction of $\{s_n\}$, we have $\lim_{n\to\infty}\mathfrak{C}=0$, and \eqref{youxian} guarantees $\mathfrak{A}<\infty$. Consequently,
$$\mathfrak{A}^2-\mathfrak{B}\mathfrak{A}+\mathfrak{C}\le0,$$
which yields
$$\mathfrak{A}\le\frac{\mathfrak{B}+\sqrt{\mathfrak{B}^2-4\mathfrak{C}}}{2}.$$
Taking the limit superior as $n\to\infty$, we arrive at
$$\limsup_{n\to\infty}\mathfrak{A}\le \mathfrak{B}.$$
Since $\mathfrak{A}$ is monotonically increasing with respect to $n$, the limit $\lim_{n\to\infty}\mathfrak{A}$ exists and satisfies
\begin{equation}\label{M1bound}
  \|\eta\|_{M_1}=\lim_{n\to\infty}\mathfrak{A}=\limsup_{n\to\infty}\mathfrak{A}\le \mathfrak{B}=\frac{2}{c_{1}} \|\eta_{s}\|_{M_{1}}<\infty,
\end{equation}
which proves $\eta\in M_1$.
Similarly, we derive the corresponding representation for $\theta(s)$:
	 \begin{equation}
		\theta(s) = s f_{3} - \smallint_{0}^{s} f_{6}\,d\tau, \label{eq:Th1.8}
	\end{equation}
and the same argument implies $\theta\in M_{2}$ and $\theta_{s}\in M_{2}$.
	
Using $\alpha_{1}=\alpha-\gamma^2 \beta$, we get from \eqref{eq:Th1.2} and \eqref{eq:Th1.4} that
	\begin{align}
		\alpha_{1}Av - \kapp_{1}A^{\theta_{1}}v + \beta\gamma(\gamma Av - Ap)
		&= -\varrho f_{2} - \smallint_{0}^{\infty} g_{1}(s)A^{\theta_{1}}\eta\,ds,\label{eq:Th1.9}\\
		\beta Ap - \gamma\beta Av - \kapp_{2}A^{\theta_{2}}p
		&= - \mu f_{4} -\smallint_{0}^{\infty} g_{2}(s)A^{\theta_{2}}\theta\,ds.\label{eq:Th1.10}
	\end{align}
Let $V\triangleq\D(A^{\frac{1}{2}})\times \D(A^{\frac{1}{2}})$ be a Banach space with standard norm
$$\|(v,p)\|_{V}=\left(\|A^{\frac{1}{2}}v\|^2+\|A^{\frac{1}{2}}p\|^2\right)^{\frac{1}{2}},~~~\forall (v,p)\in V.$$
A pair of functions
$(v,p)\in V$ is called a weak solution of \eqref{eq:Th1.9} and \eqref{eq:Th1.10}, if
$$\mathcal{B}((v,p),(\varphi,\psi)) = \mathcal{L}(\varphi,\psi),$$
where
\begin{align*}
 \mathcal{B}((v,p),(\varphi,\psi)) \triangleq&\alpha_1 \left\langle A^{\frac{1}{2}}v, A^{\frac{1}{2}}\varphi \right\rangle - \kappa_1 \left\langle A^{\frac{\theta_1}{2}}v, A^{\frac{\theta_1}{2}}\varphi \right\rangle \\
			&+ \beta \left\langle \gamma A^{\frac{1}{2}}v - A^{\frac{1}{2}}p, \gamma A^{\frac{1}{2}}\varphi - A^{\frac{1}{2}}\psi \right\rangle - \kappa_2 \left\langle A^{\frac{\theta_2}{2}}p, A^{\frac{\theta_2}{2}}\psi \right\rangle,\\
\mathcal{L}(\varphi,\psi)\triangleq&-\varrho\langle f_2,\varphi\rangle-\mu\langle f_4,\psi\rangle\\
&-\smallint_0^\infty g_1(s)\left\langle A^{\frac{\theta_1}{2}}\eta,A^{\frac{\theta_1}{2}}\varphi\right\rangle ds-\smallint_0^\infty g_2(s)\left\langle A^{\frac{\theta_2}{2}}\theta,A^{\frac{\theta_2}{2}}\psi\right\rangle ds.
\end{align*}
By \eqref{coer}, we get
\begin{align*}
\mathcal{B}((v,p),(v,p))&=\alpha_{1}\bigl\|A^{\frac{1}{2}}v\bigr\|^{2} - \kapp_{1}\bigl\|A^{\frac{\theta_1}{2}}v\bigr\|^{2}+ \beta \bigl\|\gamma A^{\frac{1}{2}}v - A^{\frac{1}{2}}p\bigr\|^{2}- \kapp_{2}\bigl\|A^{\frac{\theta_2}{2}}p\bigr\|^{2}\\
&\ge\min{\delta_1,\delta_2}\|(v,p)\|_V^2,
\end{align*}
so $\mathcal{B}$ is coercive. Since $\D(A^{\frac{1}{2}}) \hookrightarrow \D(A^{\frac{\theta_i}{2}}) \hookrightarrow H$ continuously for $i=1,2$, it is obvious that there exist two positive constant $\mathcal{M}_1$ and $\mathcal{M}_2$ such that
	\begin{align}\label{boundvb}
		|\mathcal{B}((v,p),(\varphi,\psi))|& \le \mathcal{M}_1\|(v,p)\|_{V}\,\|(\varphi,\psi)\|_{V},\notag\\
|\mathcal{L}(\varphi,\psi)|&\le\mathcal{M}_2\left(\|F\|_{\mathcal{H}}+\|\eta\|_{M_1}+\|\theta\|_{M_2}\right)\|(\varphi,\psi)\|_{V}.
	\end{align}
	By the Lax--Milgram theorem, there exists a unique \((v,p)\in V\) solving
	\(\mathcal{B}((v,p),(\varphi,\psi)) = \mathcal{L}(\varphi,\psi)\), i.e., \eqref{eq:Th1.9} and \eqref{eq:Th1.10} admits a weak solution $(v,p)\in V$. Moreover, Since we have proved $\eta,\eta_s\in M_1$, $\theta,\theta_s\in M_2$, in view of \eqref{eq:Th1.1}-\eqref{eq:Th1.4}, we get there exists $(v,u,p,q,\eta,\theta)^\top \in \D(B)$ such that \(
		B(v,u,p,q,\eta,\theta)^\top = F,
		\)  for any $F\in\mathcal{H}$, i.e., the range of $\mathcal{A}$ is the whole $\mathcal{H}$.
	
So to show $0\in\rho(B)$, we only need to show there exists $\mathcal{M}>0$ such that
		\begin{equation}\label{yjgj}
			\|(v,u,p,q,\eta,\theta)^\top\|_{\mathcal{H}} \leq \mathcal{M}\|F\|_{\mathcal{H}},
		\end{equation}where $(v,u,p,q,\eta,\theta)^\top$ is the solution of \eqref{eq:Th1.1}-\eqref{eq:Th1.6} with $F = (f_1, f_2, f_3, f_4, f_5,f_6)^\top \in \mathcal{H}$.

 By \eqref{eq:Th1.1} and \eqref{eq:Th1.3}, it is obvious that $\varrho\|u\|^{2}\le \mathcal{M}\|F\|^2$ and $ \mu\|q\|^{2} \le \mathcal{M}\|F\|^2$.

 Since $\eta_{s}=f_{1}-f_{5}$, it follows from \eqref{M1bound} that
 \begin{align*}
   \|\eta\|_{M_{1}}\le\frac{2}{c_1}\|\eta_s\|_{M_1}\le\frac{2}{c_1}\left(\|f_1\|_{M_1}+\|f_5\|_{M_1}\right)\le \mathcal{M}\left(\|A^{\frac{1}{2}}f_{1}\|+\|f_5\|_{M_1}\right)\le\mathcal{M}\|F\|_{\mathcal{H}}.
 \end{align*}
Similarly, \(\|\theta\|_{M_{2}}\le C\|F\|\).

Moreover, we get from \eqref{boundvb} that, for any $\varepsilon>0$,
\begin{align*}
&\alpha_{1}\|A^{\frac{1}{2}}v\|^{2} - \kapp_{1}\|A^{\frac{\theta_1}{2}}v\|^{2}
	+ \beta\|\gamma A^{\frac{1}{2}}v - A^{\frac{1}{2}}p\|^{2}
	- \kapp_{2}\|A^{\frac{\theta_2}{2}}p\|^{2}\\
&=\mathcal{B}((v,p),(v,p))=\mathcal{L}(v,p)\\
&\le\mathcal{M}\left(\|F\|_{\mathcal{H}}+\|\eta\|_{M_1}+\|\theta\|_{M_2}\right)\|(\varphi,\psi)\|_{V}\\
&\le\mathcal{M}\left(\|F\|_{\mathcal{H}}+\|\eta\|_{M_1}+\|\theta\|_{M_2}\right)\left(\|A^{\frac{1}{2}}v\|^2+\|A^{\frac{1}{2}}p\|^2\right)^{\frac{1}{2}}\\
&\le\varepsilon\left(\|A^{\frac{1}{2}}v\|^2+\|A^{\frac{1}{2}}p\|^2\right)+\mathcal{M}(\varepsilon)\left(\|F\|_{\mathcal{H}}+\|\eta\|_{M_1}+\|\theta\|_{M_2}\right)^2.
\end{align*}
Choosing $\varepsilon=\frac12\min\{\delta_1,\delta_2\}$ with two positive constants $\delta_1$ and $\delta_2$ being given in the proof of \eqref{coer}, by \eqref{coer}, we get
\[\varepsilon\left(\|A^{\frac{1}{2}}v\|^2+\|A^{\frac{1}{2}}p\|^2\right)\le\frac12\left(\alpha_{1}\|A^{\frac{1}{2}}v\|^{2} - \kapp_{1}\|A^{\frac{\theta_1}{2}}v\|^{2}
	+ \beta\|\gamma A^{\frac{1}{2}}v - A^{\frac{1}{2}}p\|^{2}
	- \kapp_{2}\|A^{\frac{\theta_2}{2}}p\|^{2}\right).\]
Then the above two inequalities and $\|\eta\|_{M_{1}}, \|\theta\|_{M_{2}}\le \mathcal{M}\|F\|$
\[\alpha_{1}\|A^{\frac{1}{2}}v\|^{2} - \kapp_{1}\|A^{\frac{\theta_1}{2}}v\|^{2}
	+ \beta\|\gamma A^{\frac{1}{2}}v - A^{\frac{1}{2}}p\|^{2}
	- \kapp_{2}\|A^{\frac{\theta_2}{2}}p\|^{2}\le\mathcal{M}\|F\|_{\mathcal{H}}^2.\]	
In view of the above analysis, we get \eqref{yjgj}, and the proof is complete.
\end{proof}
	
	\section{Stability analysis}
In this section, we investigate the stability of the contraction $\mathcal{C}_0$-semigroup $\{S(t)\}_{t\ge0}$ constructed in Theorem \ref{thmwellposednerss}. Our analysis relies on the two classical results stated below. Throughout both theorems, we suppose that $\{S(t)\}_{t\ge0}$ is a contraction $\mathcal{C}_0$-semigroup on a Hilbert space $(\mathcal{H},\|\cdot\|)$ with infinitesimal generator $B$.

The following characterization for polynomial stability of contraction $\mathcal{C}_0$-semigroups originates from Borichev and Tomilov \cite{borichev2010optimal}.
\begin{theorem}\label{thmBT}
Assume $i\mathbb{R}\subset \rho(B)$ with $\rho(B)$ being the resolvent set of $B$. Then the estimate
\[
\|S(t)\varphi\| \le \mathcal{M}\|B\varphi\| \,t^{-\frac1\omega},\qquad \forall \varphi\in \mathcal{D}(B),\; t\ge1,
\]
holds for some constant $\mathcal{M}>0$ and exponent $\omega>0$ if and only if
\begin{align}\label{eq:lemma1}
\sup_{\lambda\in\mathbb{R}} |\lambda|^{-\omega} \left\| (i\lambda I - B)^{-1} \right\|_{\mathcal{L}(\mathcal{H})} < \infty,
\end{align}
where $\left\| (i\lambda I - B)^{-1} \right\|_{\mathcal{L}(\mathcal{H})}$ denotes the norm of the operator $(i\lambda I - B)^{-1}$.
\end{theorem}

The subsequent criterion for exponential stability of contraction $\mathcal{C}_0$-semigroups is established by Pr\"us \cite{pruss1984spectrum}.
\begin{theorem}\label{thmG}
The semigroup $\{S(t)\}_{t\ge0}$ is exponentially stable if and only if
\(
i\mathbb{R}\subset \rho(B)
\)
and
\begin{align}\label{eq:lemma2}
\sup_{\lambda\in\mathbb{R}} \left\| (i\lambda I - B)^{-1} \right\|_{\mathcal{L}(\mathcal{H})} < \infty.
\end{align}
\end{theorem}
	\subsection{Polynomially stability with $\theta_1<1$ or $\theta_2<1$}
The main result established in this part is stated in the following theorem.
\begin{theorem}\label{thmpoly}
Suppose that assumptions (A1)--(A3) are satisfied. If either $\theta_1<1$ or $\theta_2<1$, then the semigroup $S(t)$ generated by the problem \eqref{modelmain} is polynomially stable with the decay rate $t^{-\frac{1}{2-2\theta_{0}}}$ for smooth initial data. Specifically, for any initial state $X_0 \in \mathcal{D}({B})$, the estimate
\begin{align}
\|S(t)X_0\|_{\mathcal{H}} \leq C t^{-\frac{1}{2-2\theta_0}} \|BX_0\|_{\mathcal{H}}, \quad \forall t \geq 1,
\end{align}
holds, where $C > 0$ is a positive constant and $\theta_{0} = \min\{\theta_1,\theta_2\}$.
\end{theorem}
By virtue of Theorem \ref{thmBT}, the above theorem can be readily deduced from the following two lemmas.
\begin{lemma}\label{lemthm1-1}
Under the assumptions of Theorem \ref{thmpoly}, the inclusion $i\mathbb{R}\subset\rho({B})$ holds.
\end{lemma}
\begin{lemma}\label{lemthm1-2}
Under the assumptions of Theorem \ref{thmpoly}, one has
$$\sup_{\lambda\in\mathbb{R}} |\lambda|^{-2(1-\theta_0)} \left\| (i\lambda I - {B})^{-1} \right\|_{\mathcal{L}(\mathcal{H})} < \infty.$$
\end{lemma}
The rest of this part is devoted to the proofs of the above two lemmas.
	\begin{proof} [Proof of Lemma \ref{lemthm1-1}]
We prove this lemma by contradiction argument. Set
		\[
		\hat{\lambda} \triangleq \sup\{R>0 : [-Ri,Ri]\subset\varrho(B)\}.
		\]
If \(i\R\not\subset\varrho(B)\), since \(0\in\varrho(B)\) and \(\varrho(B)\) is open,  we get $\hat{\lambda}<\infty$. Then by the Banach-Steinhaus theorem, there exists a sequence
		\(X_{n}=(v_{n},u_{n},p_{n},q_{n},\eta_{n},\theta_{n})\in\D(B)\),
		\(\|X_{n}\|_{\Hil}=1\), with \(|\lambda_{n}|<\hat{\lambda}\), \(\lambda_{n}\to\hat{\lambda}\),
		such that \((i\lambda_{n}-B)X_{n}=F_{n}\triangleq (f_n^1,f_n^2,f_n^3,f_n^4,f_n^5,f_n^6)^\top\to0\) in \(\Hil\), i.e.,
		\begin{align}
				&i\lambda_{n}v_{n} - u_{n} = f_{n}^{1} \to 0 \quad \text{in } D(A^{\frac{1}{2}}),\label{eq:Th3.2.1}\\
				&i\lambda_{n}u_{n} - \frac{1}{\varrho}\Bigl(-\alpha Av_{n} + \gamma\beta Ap_{n}
				+ \kapp_{1}A^{\theta_{1}}v_{n} - \smallint_{0}^{\infty}g_{1}(s)A^{\theta_{1}}\eta_{n}\,ds\Bigr)
				= f_{n}^{2} \to 0\quad \text{in } H,  \label{eq:Th3.2.2}\\
				&i\lambda_{n}p_{n} - q_{n} = f_{n}^{3} \to 0\quad \text{in } D(A^{\frac{1}{2}}), \label{eq:Th3.2.3}\\
				&i\lambda_{n}q_{n} - \frac{1}{\mu}\Bigl(-\beta Ap_{n} + \gamma\beta Av_{n}
				+ \kapp_{2}A^{\theta_{2}}p_{n} - \smallint_{0}^{\infty}g_{2}(s)A^{\theta_{2}}\theta_{n}\,ds\Bigr)
				= f_{n}^{4} \to 0\quad \text{in } H, \label{eq:Th3.2.4}\\
				&i\lambda_{n}\eta_{n} + \eta_{n,s} - u_{n} = f_{n}^{5} \to 0\quad \text{in } \M{1}, \label{eq:Th3.2.5}\\
				&i\lambda_{n}\theta_{n} + \theta_{n,s} - q_{n} = f_{n}^{6} \to 0 \quad \text{in } \M{2}. \label{eq:Th3.2.6}
		\end{align}
		
To get contradiction, we will prove \(\|X_{n}\|_{\mathcal{H}}=o(1)\), i.e., $\lim_{n\to\infty}\|X_{n}\|_{\mathcal{H}}=0$.

	By \eqref{disB}, we get
		\[
		\mathrm{Re}\langle (i\lambda_{n}-B)X_{n},X_{n}\rangle
		= -\frac{1}{2}\smallint_{0}^{\infty}g_{1}'(s)\|A^{\frac{\theta_1}{2}}\eta_{n}\|^{2}ds
		-\frac{1}{2}\smallint_{0}^{\infty}g_{2}'(s)\|A^{\frac{\theta_2}{2}}\theta_{n}\|^{2}ds.
		\]
Since $\|(i\lambda_{n}-B)X_{n}\|_{\mathcal{H}}=o(1)$ and $\|X_{n}\|_{\mathcal{H}}=1$,  it follows from the above inequality that
\[-\frac{1}{2}\smallint_{0}^{\infty}g_{1}'(s)\|A^{\frac{\theta_1}{2}}\eta_{n}\|^{2}ds
		-\frac{1}{2}\smallint_{0}^{\infty}g_{2}'(s)\|A^{\frac{\theta_2}{2}}\theta_{n}\|^{2}ds=o(1).\]
Then using the assumptions (A1) and (A2), we get
		 \begin{align}
		\|\eta_{n}\|_{M_{1}}=o(1),  \label{eq:Th3.2.1c}\\
		\|\theta_{n}\|_{M_{2}}=o(1). \label{eq:Th3.2.2c}
		\end{align}
		
Since $\|A^{\frac{1}{2}}f_n^1\|=o(1)$, $\|A^{\frac{1}{2}}v_n\|\le\|X_n\|_{\mathcal{H}}=1$ and $|\lambda_n|\le\hat\lambda<\infty$, it follows from \eqref{eq:Th3.2.1} that $\|A^{\frac{1}{2}}u_n\|$ is bounded. Then by $\D(A^{\frac{1}{2}})\hookrightarrow\D(A^{\frac{\theta_i}{2}})$, we get $\|A^{\frac{\theta_i}{2}}u_n\|$ is bounded for $i=1,2$. Similarly, we have $\|A^{\frac{\theta_i}{2}}q_n\|$ is bounded for $i=1,2$.

	By taking the \(M_{1}\)-inner product of \eqref{eq:Th3.2.5} with \(u_{n}\), we have
\begin{equation}\label{fg1}
 \langle i\lambda_{n}\eta_{n},u_{n}\rangle_{M_{1}}
		+ \langle \eta_{n,s},u_{n}\rangle_{M_{1}}
		- \langle u_{n},u_{n}\rangle_{M_{1}} = \langle f_n^5, u_n\rangle_{M_1}.
\end{equation}
For the term on the right-hand side of \eqref{fg1}, since $\|A^{\frac{\theta_1}{2}}u_n\|$ is bounded and $\|f_n^5\|_{M_1}=o(1)$, we have
$$|\langle f_n^5, u_n\rangle_{M_1}|\le\smallint_0^\infty g_1(s)\left|\left\langle A^{\frac{\theta_1}{2}}f_n^5,A^{\frac{\theta_1}{2}}u_n\right\rangle\right| ds\le\sqrt{\kappa_1}\|A^{\frac{\theta_1}{2}}u_n\|\|f_n^5\|_{M_1}=o(1).$$
For the first term on the left-hand side of \eqref{fg1}, since $|\lambda_n|\le\hat\lambda<\infty$, $\|A^{\frac{\theta_1}{2}} u_n \| $ is bounded, and $\left\| \eta_n \right\|_{\M{1}} =o(1)$, we have
		\begin{align*}
			\left| \langle i\lambda_n \eta_n, u_n \rangle_{\M{1}} \right|
			&\leq |\lambda_n| \smallint_0^\infty g_1(s) \left| \left\langle A^{\frac{\theta_1}{2}} \eta_n, A^{\frac{\theta_1}{2}} u_n \right\rangle \right| ds \\
			&\leq |\lambda_n| \left\| A^{\frac{\theta_1}{2}} u_n \right\| \smallint_0^\infty g_1(s) \left\| A^{\frac{\theta_1}{2}} \eta_n \right\| ds \leq |\lambda_n| \sqrt{\kappa_1} \left\| A^{\frac{\theta_1}{2}} u_n \right\| \left\| \eta_n \right\|_{\M{1}} =o(1).
		\end{align*}
For the second term on the left-hand side of \eqref{fg1}, since the assumption (A2), $\|A^{\frac{\theta_1}{2}} u_n \| $ is bounded, and $\left\| \eta_n \right\|_{\M{1}} =o(1)$, we have
		\begin{align*}
			\left| \langle \eta_{n,s}, u_n \rangle_{\M{1}} \right|
			&= \left| \smallint_0^\infty g_1(s) \left\langle A^{\frac{\theta_1}{2}} \eta_{n,s}, A^{\frac{\theta_1}{2}} u_n \right\rangle ds \right| = \left| - \smallint_0^\infty g_1'(s) \left\langle A^{\frac{\theta_1}{2}} \eta_n, A^{\frac{\theta_1}{2}} u_n \right\rangle ds \right| \\
			&\leq c_0 \left| \smallint_0^\infty g_1(s) \left\langle A^{\frac{\theta_1}{2}} \eta_n, A^{\frac{\theta_1}{2}} u_n \right\rangle ds \right| \leq c_0 \left\| A^{\frac{\theta_1}{2}} u_n \right\| \sqrt{\kappa_1} \left\| \eta_n \right\|_{\M{1}} =o(1).
		\end{align*}
		Hence by \eqref{fg1} and the above analysis, it follows
$$\kapp_{1}\|A^{\frac{\theta_1}{2}}u_{n}\|^{2}=o(1),$$
which, together with $\D(A^{\frac{\theta_1}{2}}) \hookrightarrow H$, implies that
		 \begin{align} \label{eq:Th3.2.3c}
		\|u_{n}\|=o(1).
		\end{align}
	In the same way, by taking the \(M_{2}\)-inner product of \eqref{eq:Th3.2.6} with \(q_{n}\), one can get
	    \begin{align} \label{eq:Th3.2.4c}
	   	\|q_{n}\|=o(1).
	   \end{align}
		
	Substituting \eqref{eq:Th3.2.1} into \eqref{eq:Th3.2.2}, taking $H$-inner products with \(v_{n}\), we obtain
		\begin{equation} \label{eq:Th3.2.7}
		\begin{split}
		&-\varrho \lambda_n^2 \norm{v_n}^2 + \alpha_1 \norm{A^{\frac{1}{2}} v_n}^2 + \beta \gamma^2 \norm{A^{\frac{1}{2}} v_n}^2 - \beta \inner{A^{\frac{1}{2}} p_n}{\gamma A^{\frac{1}{2}} v_n} \\
		&- \kappa_1 \norm{A^{\frac{\theta_1}{2}} v_n}^2 + \smallint_0^\infty g_1(s) \inner{A^{\frac{\theta_1}{2}} \eta_n}{A^{\frac{\theta_1}{2}} v_n} ds =i\lambda_n\varrho\langle f_n^1,v_n\rangle+\varrho\langle f_n^2,v_n\rangle,
		\end{split}
	    \end{equation}
where we have used $\alpha=\alpha_1+\beta\gamma^2$.	 Substituting \eqref{eq:Th3.2.3} into \eqref{eq:Th3.2.4}, taking inner products with \(p_{n}\), we obtain
		\begin{equation} \label{eq:Th3.2.8}
		\begin{split}
		&-\mu \lambda_n^2 \norm{p_n}^2 + \beta \norm{A^{\frac{1}{2}} p_n}^2 - \beta \inner{\gamma A^{\frac{1}{2}} v_n}{A^{\frac{1}{2}} p_n} \\
		&- \kappa_2 \norm{A^{\frac{\theta_2}{2}} p_n}^2 + \smallint_0^\infty g_2(s) \inner{A^{\frac{\theta_2}{2}} \theta_n}{A^{\frac{\theta_2}{2}} p_n} ds =i\lambda_n\mu\langle f_n^3, p_n\rangle+\mu\langle f_n^4, p_n\rangle.
	    	\end{split}
        \end{equation}
		Adding \eqref{eq:Th3.2.7} and \eqref{eq:Th3.2.8} together, we get
		\begin{align*}
				&-\varrho \lambda_n^2 \|v_n\|^2 + \alpha_1 \|A^{\frac{1}{2}} v_n\|^2 - \kappa_1 \|A^{\frac{\theta_1}{2}} v_n\|^2 + \beta \left\| \gamma A^{\frac{1}{2}} v_n - A^{\frac{1}{2}} p_n \right\|^2 - \kappa_2 \|A^{\frac{\theta_2}{2}} p_n\|^2 \\
				&- \mu \lambda_n^2 \|p_n\|^2 + \smallint_0^\infty g_1(s) \left\langle A^{\frac{\theta_1}{2}} \eta_n(s), A^{\frac{\theta_1}{2}} v_n \right\rangle ds + \smallint_0^\infty g_2(s) \left\langle A^{\frac{\theta_2}{2}} \theta_n(s), A^{\frac{\theta_2}{2}} p_n \right\rangle ds\\
& =i\lambda_n\varrho\langle f_n^1,v_n\rangle+\varrho\langle f_n^2,v_n\rangle+i\lambda_n\mu\langle f_n^3, p_n\rangle+\mu\langle f_n^4, p_n\rangle.
		\end{align*}
		Using \eqref{eq:Th3.2.1c}, \eqref{eq:Th3.2.2c}, and the boundedness of  $\|A^{\frac{\theta_1}{2}} v_n\|$ and $\|A^{\frac{\theta_2}{2}} p_n\|$, we have
		\begin{align*}
			\left| \smallint_0^\infty g_1(s) \left\langle A^{\frac{\theta_1}{2}} \eta_n(s), A^{\frac{\theta_1}{2}} v_n \right\rangle ds \right|
			&\leq \|A^{\frac{\theta_1}{2}} v_n\| \sqrt{\kappa_1} \|\eta_n\|_{\M{1}}=o(1), \\
			\left| \smallint_0^\infty g_2(s) \left\langle A^{\frac{\theta_2}{2}} \theta_n(s), A^{\frac{\theta_2}{2}} p_n \right\rangle ds \right|
			&\leq \|A^{\frac{\theta_2}{2}} p_n\| \sqrt{\kappa_2} \|\theta_n\|_{\M{2}}=o(1).
		\end{align*}
Moreover, since $|\lambda_n|<\hat\lambda<\infty$, $|\langle f_n^i,v_n\rangle|\le\|f_n^i\|\|v_n\|$, $\|f_n^i\|=o(1)$, and $\|v_n\|$ is bounded for all $i=1,2,3,4$, we get
$$\left|i\lambda_n\varrho\langle f_n^1,v_n\rangle+\varrho\langle f_n^2,v_n\rangle+i\lambda_n\mu\langle f_n^3, p_n\rangle+\mu\langle f_n^4, p_n\rangle\right|=o(1).$$
		Therefore, combining \eqref{eq:Th3.2.3c} and \eqref{eq:Th3.2.4c}, we get
		\begin{align} \label{eq:Th3.2.5c}
		\alpha_{1}\|A^{\frac{1}{2}}v_{n}\|^{2} - \kapp_{1}\|A^{\frac{\theta_1}{2}}v_{n}\|^{2}
		+ \beta\|\gamma A^{\frac{1}{2}}v_{n} - A^{\frac{1}{2}}p_{n}\|^{2}
		- \kapp_{2}\|A^{\frac{\theta_2}{2}}p_{n}\|^{2} =o(1).
	    \end{align}
		
Combing \eqref{eq:Th3.2.1c}, \eqref{eq:Th3.2.2c}, \eqref{eq:Th3.2.3c}, \eqref{eq:Th3.2.4c}, \eqref{eq:Th3.2.5c}, we get $\|X_n\|_{\Hil}=o(1)$, which contradicts \(\|X_{n}\|_{\Hil}=1\).
Hence the desired result follows.
	\end{proof}

	\begin{proof}[Proof of Lemma \ref{lemthm1-2}.]
	    The proof by contradiction is still employed. If the conclusion is not true, then there exist \(X_{n}=(v_{n},u_{n},p_{n},q_{n},\eta_{n},\theta_{n})\in\D(B)\), \(\|X_{n}\|_{\Hil}=1\), and \(\lambda_{n}\to\infty\), such that \(\lambda_{n}^{2w}(i\lambda_{n}-B)X_{n}=F_{n}\to0\) in \(\Hil\), where
\[w\triangleq1-\theta_0.\]
That is
		\begin{align}
			&\lambda_{n}^{2w}(i\lambda_{n}v_{n} - u_{n}) = f_{n}^{1}\to0, \quad \text{in } D(A^{\frac{1}{2}})\label{eq:Th3.2..1}\\
			&\lambda_{n}^{2w}\Bigl(i\lambda_{n}u_{n} - \frac{1}{\varrho}
			\bigl(-\alpha Av_{n} + \gamma\beta Ap_{n} + \kapp_{1}A^{\theta_{1}}v_{n}
			- \smallint_{0}^{\infty}g_{1}(s)A^{\theta_{1}}\eta_{n}ds\bigr)\Bigr) = f_{n}^{2}\to0, \quad \text{in } H\label{eq:Th3.2..2}\\
			&\lambda_{n}^{2w}(i\lambda_{n}p_{n} - q_{n}) = f_{n}^{3}\to0, \quad \text{in } D(A^{\frac{1}{2}})\label{eq:Th3.2..3}\\
			&\lambda_{n}^{2w}\Bigl(i\lambda_{n}q_{n} - \frac{1}{\mu}
			\bigl(-\beta Ap_{n} + \gamma\beta Av_{n} + \kapp_{2}A^{\theta_{2}}p_{n}
			- \smallint_{0}^{\infty}g_{2}(s)A^{\theta_{2}}\theta_{n}ds\bigr)\Bigr) = f_{n}^{4}\to0, \quad \text{in } H\label{eq:Th3.2..4}\\
			&\lambda_{n}^{2w}(i\lambda_{n}\eta_{n} + \eta_{n,s} - u_{n}) = f_{n}^{5}\to0, \quad \text{in } \M{1}\label{eq:Th3.2..5}\\
			&\lambda_{n}^{2w}(i\lambda_{n}\theta_{n} + \theta_{n,s} - q_{n}) = f_{n}^{6}\to0.\quad \text{in } \M{2}\label{eq:Th3.2..6}
		\end{align}
	
By \eqref{disB}, we get
		\[
		\mathrm{Re}\langle\lambda_{n}^{2w} (i\lambda_{n}-B)X_{n},X_{n}\rangle
		= -\frac{1}{2}\lambda_{n}^{2w}\smallint_{0}^{\infty}g_{1}'(s)\|A^{\frac{\theta_1}{2}}\eta_{n}\|^{2}ds
		-\frac{1}{2}\lambda_{n}^{2w}\smallint_{0}^{\infty}g_{2}'(s)\|A^{\frac{\theta_2}{2}}\theta_{n}\|^{2}ds.
		\]
Since $\|\lambda_{n}^{2w}(i\lambda_{n}-B)X_{n}\|_{\mathcal{H}}=o(1)$ and $\|X_{n}\|_{\mathcal{H}}=1$,  it follows from the above inequality that
\[-\frac{1}{2}\smallint_{0}^{\infty}g_{1}'(s)\|A^{\frac{\theta_1}{2}}\eta_{n}\|^{2}ds
		-\frac{1}{2}\smallint_{0}^{\infty}g_{2}'(s)\|A^{\frac{\theta_2}{2}}\theta_{n}\|^{2}ds=\lambda_{n}^{-2w}o(1).\]
		Then using the assumptions (A1) and (A2), we get
		 \begin{align}
			\|\eta_{n}\|_{M_{1}} = |\lambda_{n}|^{-w}o(1),\label{eq:Th3.2..1c}\\
			\|\theta_{n}\|_{M_{2}} = |\lambda_{n}|^{-w}o(1). \label{eq:Th3.2..2c}
		\end{align}

		Multiplying \eqref{eq:Th3.2..5} by \(\lambda_{n}^{-2w}\), we get
		\begin{align} \label{eq:Th3.2..7}
		i\lambda_{n}\eta_{n} + \eta_{n,s} - u_{n} = \lambda_{n}^{-2w} f_n^5 \to0, \quad \text{in } \M{1}
		\end{align}
		Taking the \( \M{1} \)-inner product of equation \eqref{eq:Th3.2..7} with \( u_n \), we obtain
\begin{equation}\label{zj1}
\inner{i\lambda_n \eta_n}{u_n}_{\M{1}} + \inner{\eta_{n,s}}{u_n}_{\M{1}} - \inner{u_n}{u_n}_{\M{1}} =\inner{\lambda_{n}^{-2w} f_n^5}{u_n}_{\M{1}}.
\end{equation}
		For the first term of \eqref{zj1}, by using \eqref{eq:Th3.2..1c}, we get
		\begin{align} \label{eq:Th3.2..8}
			\left| \inner{i\lambda_n \eta_n}{u_n}_{\M{1}} \right|
			&= \left| \smallint_0^\infty i\lambda_n g_1(s) \inner{A^{\frac{\theta_1}{2}} \eta_{n}}{A^{\frac{\theta_1}{2}} u_n} ds \right| \nonumber \leq |\lambda_n| \left\| A^{\frac{\theta_1}{2}} u_n \right\| \smallint_0^\infty g_1(s) \left\| A^{\frac{\theta_1}{2}} \eta_n \right\| ds \nonumber \\
			&\leq |\lambda_n| \sqrt{\kappa_1} \norm{\eta_n}_{\M{1}} \norm{A^{\frac{\theta_1}{2}} u_n} = |\lambda_n|^{1-w} \norm{A^{\frac{\theta_1}{2}} u_n} \oone.
		\end{align}
		For the second term of \eqref{zj1}, by (A2) and \eqref{eq:Th3.2..1c}, it follows that
		\begin{align}\label{eq:Th3.2..9}
			\left| \inner{\eta_{n,s}}{u_n}_{\M{1}} \right|
			&= \left| \smallint_0^\infty g_1(s) \inner{A^{\frac{\theta_1}{2}} \eta_{n,s}}{A^{\frac{\theta_1}{2}} u_n} ds \right| \nonumber = \left| - \smallint_0^\infty g_1'(s) \inner{A^{\frac{\theta_1}{2}} \eta_n}{A^{\frac{\theta_1}{2}} u_n} ds \right| \\
			&\leq c_0 \left| \smallint_0^\infty g_1(s) \inner{A^{\frac{\theta_1}{2}} \eta_n}{A^{\frac{\theta_1}{2}} u_n} ds \right| \nonumber
			\leq c_0 \sqrt{\kappa_1} \norm{\eta_n}_{\M{1}} \norm{A^{\frac{\theta_1}{2}} u_n} \\
			&= |\lambda_n|^{-w} \norm{A^{\frac{\theta_1}{2}} u_n} \oone.
		\end{align}	
		Similarly, for the right-hand side of \eqref{zj1}, we have
		\begin{align}\label{eq:Th3.2..10}
			\left| \lambda_n^{-2w} \inner{f_n^5}{u_n}_{\M{1}} \right| \leq |\lambda_n|^{-2w}\sqrt{\kappa_1} \norm{f_n^5}_{\M{1}} \norm{A^{\frac{\theta_1}{2}} u_n} = |\lambda_n|^{-2w} \norm{A^{\frac{\theta_1}{2}} u_n} \oone.
		\end{align}
		Combining \eqref{zj1}, \eqref{eq:Th3.2..8}, \eqref{eq:Th3.2..9}, and \eqref{eq:Th3.2..10} , we get
		\[
		\kappa_1 \norm{A^{\frac{\theta_1}{2}} u_n}^2 \leq |\lambda_n|^{1-w} \norm{A^{\frac{\theta_1}{2}} u_n} \oone +|\lambda_n|^{-w} \norm{A^{\frac{\theta_1}{2}} u_n} \oone + |\lambda_n|^{-2w} \norm{A^{\frac{\theta_1}{2}} u_n} \oone,
		\]
		which is simplified to
		\begin{align} \label{eq:Th3.2..3c}
			\norm{A^{\frac{\theta_1}{2}} u_n} \leq |\lambda_n|^{1-w} \oone.
		\end{align}
		
		Taking the \( H \)-inner product of equation \eqref{eq:Th3.2..7} with \( A^{\theta_1-1} u_n \), then multiplying both sides by \( g_1(s) \) and integrating with respect to \( s \) from \( 0 \) to \( \infty \), we obtain
		\begin{align}\label{zj2}
			&\smallint_0^\infty i\lambda_n g_1(s) \inner{A^{\frac{\theta_1-1}{2}} \eta_n}{A^{\frac{\theta_1-1}{2}} u_n} ds + \smallint_0^\infty g_1(s) \inner{A^{\frac{\theta_1-1}{2}} \eta_{n,s}}{A^{\frac{\theta_1-1}{2}} u_n} ds \notag\\
			&- \smallint_0^\infty g_1(s) \norm{A^{\frac{\theta_1-1}{2}} u_n}^2 ds  =\lambda_n^{-2w} \smallint_0^\infty g_1(s) \inner{A^{\frac{\theta_1-1}{2}} f_n^5}{A^{\frac{\theta_1-1}{2}} u_n} ds.
		\end{align}
For the first term of \eqref{zj2}, by \eqref{eq:Th3.2..2}, we have
		\begin{align*}
&\left| \smallint_0^\infty i\lambda_n g_1(s) \inner{A^{\frac{\theta_1-1}{2}} \eta_n}{A^{\frac{\theta_1-1}{2}} u_n} ds \right|
			= \left| \smallint_0^\infty g_1(s) \inner{A^{\theta_1-1} \eta_n}{i\lambda_n u_n} ds \right|\\
			&=\Bigg| \frac{1}{\varrho} \Big[ \smallint_0^\infty g_1(s) \inner{A^{\theta_1-1} \eta_n} {-\alpha A v_n}ds  + \smallint_0^\infty g_1(s) \inner{A^{\theta_1-1} \eta_n }{\gamma \beta A p_n}ds  \\
			&\quad + \smallint_0^\infty g_1(s) \inner{A^{\theta_1-1} \eta_n}{\kappa_1 A^{\theta_1} v_n}ds - \smallint_0^\infty g_1(s) \inner{A^{\theta_1-1} \eta_n}{ \smallint_0^\infty g_1(s) A^{\theta_1} \eta_n ds }ds \Big]\\
			&\quad + \frac{1}{\lambda_n^{2w}} \smallint_0^\infty g_1(s) \inner{A^{\theta_1-1} \eta_n}{ f_n^2}ds  \Bigg|.
		\end{align*}
		Since $\D(A^{\frac{\theta_1}{2}}) \hookrightarrow \D(A^{\frac{2\theta_1-1}{2}})$, we obtain, for some constant $C>0$,
		\begin{align*}
			\Bigg|&\smallint_0^\infty g_1(s) \inner{A^{\theta_1-1} \eta_n}{i\lambda_n u_n} ds\Bigg| \\
			&\leq \frac{1}{\varrho} \Bigg[ \alpha \left| \smallint_0^\infty g_1(s) \inner{A^{\frac{\theta_1}{2}} \eta_n}{A^{\frac{\theta_1}{2}} v_n} ds \right| + \gamma \beta \left| \smallint_0^\infty g_1(s) \inner{A^{\frac{\theta_1}{2}} \eta_n}{A^{\frac{\theta_1}{2}} p_n} ds \right| \\
			&\quad + \kappa_1 \left| \smallint_0^\infty g_1(s) \inner{A^{\frac{\theta_1}{2}} \eta_n}{A^{\frac{\theta_1}{2}} v_n} ds \right| + \kappa_1 \smallint_0^\infty g_1(s) \norm{A^{\frac{2\theta_1-1}{2}} \eta_n}^2 ds \Bigg]\\
			&\quad + \frac{1}{|\lambda_n|^{2w}} \left| \smallint_0^\infty g_1(s) \inner{A^{\frac{\theta_1-1}{2}} \eta_n}{A^{\frac{\theta_1-1}{2}} f_n^2} ds \right|  \\
			&\leq \frac{1}{\varrho} \Bigg[ \alpha \norm{A^{\frac{\theta_1}{2}} v_n} \sqrt{\kappa_1} \norm{\eta_n}_{\M{1}} + \gamma \beta \norm{A^{\frac{\theta_1}{2}} p_n} \sqrt{\kappa_1} \norm{\eta_n}_{\M{1}} + \kappa_1 \norm{A^{\frac{\theta_1}{2}} v_n} \sqrt{\kappa_1} \norm{\eta_n}_{\M{1}}   \\
			&\quad+ \kappa_1 C^2 \norm{\eta_n}_{\M{1}}^2 \Bigg]+ \frac{C^2}{|\lambda_n|^{2w}} \norm{f_n^2} \sqrt{\kappa_1} \norm{\eta_n}_{\M{1}}  \\
&=\left({|\lambda_n|^{-w}} + {|\lambda_n|^{-3w}}\right)\oone={|\lambda_n|^{-w}}\oone,\end{align*}
		where we have used in \eqref{eq:Th3.2..1c}, \eqref{eq:Th3.2..2c}, and  the boundedness of $\left\| A^{\frac{\theta_1}{2}} v_n \right\|$, $\left\| A^{\frac{\theta_1}{2}} p_n \right\|$ and $\left\| {f}_n^2 \right\|$.

Combining the above two inequalities, we get
\begin{equation}\label{zj3}
\left| \smallint_0^\infty i\lambda_n g_1(s) \inner{A^{\frac{\theta_1-1}{2}} \eta_n}{A^{\frac{\theta_1-1}{2}} u_n} ds \right|={|\lambda_n|^{-w}}\oone.
\end{equation}

		For the second term of \eqref{zj2}, since $H \hookrightarrow \D(A^{{(\theta_{1}-1)}/{2}})$ for $\theta_1\le 1$, by the Assumption (A2), \eqref{eq:Th3.2..1c}, and the boundedness of $\|u_n\|$, we have, for some constant $C>0$,
		\begin{align}\label{zj4}
			&\left| \smallint_0^\infty g_1(s) \inner{A^{\frac{\theta_1-1}{2}} \eta_{n,s}}{A^{\frac{\theta_1-1}{2}} u_n} ds \right|
			= \left| -\smallint_0^\infty g_1'(s) \inner{A^{\frac{\theta_1-1}{2}} \eta_{n}}{A^{\frac{\theta_1-1}{2}} u_n} ds \right|\notag\\
			&\leq c_0 \| A^{\frac{\theta_1-1}{2}} u_n\| \smallint_{0}^{\infty} g_1(s)\| A^{\frac{\theta_1-1}{2}} \eta_n\|
			\leq C \| u_n\| \sqrt{\kappa_1} \norm{\eta_n}_{\M{1}}={|\lambda_n|^{-w}}\oone.
		\end{align}

By the same way, the right-hand side term of \eqref{zj2} can be estimated as
\begin{equation}\label{zj5}
\left| \frac{1}{\lambda_n^{2w}} \smallint_0^\infty g_1(s) \inner{A^{\frac{\theta_1-1}{2}} f_n^5}{A^{\frac{\theta_1-1}{2}} u_n} ds \right|={|\lambda_n|^{-2w}}\oone.
\end{equation}

		Substituting \eqref{zj3}--\eqref{zj5} into \eqref{zj2}, we obtain
		\begin{align} \label{eq:Th3.2..4c}
			\norm{A^{\frac{\theta_1-1}{2}} u_n} ={|\lambda_n|^{-w}}\oone.
		\end{align}

		Therefore, it follows from \eqref{eq:Th3.2..3c} and \eqref{eq:Th3.2..4c} that
		\[
		\|u_n\| \le \|A^{\frac{\theta_1}{2}}u_n\|^{1-\theta_{1}}\,
		\|A^{\frac{\theta_1-1}{2}}u_n\|^{\theta_{1}}=|\lambda_{n}|^{(1-w)(1-\theta_{1}) - w\theta_{1}} o(1).
		\]
Since $w=1-\theta_0\ge 1-\theta_1$, we get $(1-w)(1-\theta_{1}) - w\theta_{1}\le0$, then the above implies
\begin{equation}\label{zj6}
  \|u_n\|=o(1).
\end{equation}

In the same way as the proof of \eqref{zj6}, since $w=1-\theta_0\ge 1-\theta_2$, we get
\begin{equation}\label{zj7}
  \|q_n\|=o(1).
\end{equation}

By the same argument as in the proof of Lemma \ref{lemthm1-1}, it still holds that
\begin{equation}\label{zj8}
  \alpha_{1}\|A^{\frac{1}{2}}v_{n}\|^{2} - \kapp_{1}\|A^{\frac{\theta_1}{2}}v_{n}\|^{2}
			+ \beta\|\gamma A^{\frac{1}{2}}v_{n} - A^{\frac{1}{2}}p_{n}\|^{2}
			- \kapp_{2}\|A^{\frac{\theta_2}{2}}p_{n}\|^{2} =o(1),
\end{equation}

Thus, the estimates \eqref{zj6}--\eqref{zj8} imply \(\|X_{n}\|_{\Hil}=o(1)\),
	which	contradicts \(\|X_{n}\|_{\Hil}=1\). The proof is completed.
	\end{proof}

\begin{theorem}\label{thmopt}
	Suppose that assumptions (A1)--(A3) hold and that $\theta_1<1$ or $\theta_2<1$. If the memory kernels $g_j$, $j=1,2$,  are of exponential type, then the polynomial decay rate $t^{-\frac{1}{2-2\theta_0}}$ obtained in Theorem \ref{thmpoly}, where $\theta_0=\min\{\theta_1,\theta_2\}$, is optimal.
\end{theorem}	
\begin{proof}
We consider memory kernels decaying exponentially
\[
g_{1}(s) = m_{1}e^{-\mu_{1}s},\qquad g_{2}(s) = m_{2}e^{-\mu_{2}s},\qquad
m_{1},m_{2}>0,\;\; \mu_{1},\mu_{2}>0.
\]
Let \(Ae_{n} = \xi_{n}e_{n}\), \(\|e_{n}\|=1\), \(n\in\mathbb{N}\).

We consider
\[
F_{n} = (0, c_{1}e_{n}, 0, c_{2}e_{n}, 0, 0)
\]
\(X=(v,u,p,q,\eta,\theta)\) is the solution of the system \((i\lambda I-B)X = F_n\) can be considered as followed.
\begin{align}
	i\lambda v - u = 0, \label{eq:4.1}\\
	\varrho i\lambda u + \alpha Av - \gamma\beta Ap
	- \kapp_{1} A^{\theta_{1}}v
	+ \smallint_{0}^{\infty}g_{1}(s) A^{\theta_{1}}\eta(s)\,ds = \varrho c_{1}e_{n}, \label{eq:4.2}\\
	i\lambda p - q = 0, \label{eq:4.3}\\
	\mu i\lambda q + \beta Ap - \gamma\beta Av
	- \kapp_{2} A^{\theta_{2}}p
	+ \smallint_{0}^{\infty}g_{2}(s) A^{\theta_{2}}\theta(s)\,ds = \mu c_{2}e_{n}, \label{eq:4.4}\\
	i\lambda\eta + \eta_{s} - u = 0, \label{eq:4.5}\\
	i\lambda\theta + \theta_{s} - q = 0. \label{eq:4.6}
\end{align}

Using the equation \eqref{eq:4.1}, \eqref{eq:4.5} and \(\eta(0)=0\) we get
\[\eta(s) = v(1-e^{-i\lambda s})\]
Similarly, using the equation \eqref{eq:4.2}, \eqref{eq:4.6} and \(\theta(0)=0\) we get
\[\theta(s) = p(1-e^{-i\lambda s})\]
Substituting \(\eta,\theta\) into \eqref{eq:4.2} and \eqref{eq:4.4} and using \(A^{k}e_{n}=\xi_{n}^{k}e_{n}\), we get
\begin{align*}
	\varrho\lambda^{2}v - \alpha Av + \gamma\beta Ap
	+ \Bigl(\smallint_{0}^{\infty}g_{1}(s)e^{-i\lambda s}ds\Bigr) A^{\theta_{1}}v
	&= -\varrho c_{1}e_{n},\\
	\mu\lambda^{2}p - \beta Ap + \gamma\beta\ Av
	+ \Bigl(\smallint_{0}^{\infty}g_{2}(s)e^{-i\lambda s}ds\Bigr) A^{\theta_{2}}p
	&= -\mu c_{2}e_{n}.
\end{align*}
We are going to look for solutions of form
\[\begin{cases}
	v = \tau_1 e_n \\
	p = \tau_2 e_n
\end{cases}, \quad \tau_1, \tau_2 \in \mathbb{C}
\]
We will have the system
\[
\begin{cases}
	\left( \varrho \lambda^2 - \alpha \xi_n + \smallint_0^\infty g_1(s) e^{-i\lambda s} ds \cdot \xi_n^{\theta_1} \right) \tau_1 + \gamma \beta \xi_n \cdot \tau_2 = -\varrho c_1, \\
	\gamma \beta \xi_n \cdot \tau_1 + \left( \mu \lambda^2 - \beta \xi_n + \smallint_0^\infty g_2(s) e^{-i\lambda s} ds \cdot \xi_n^{\theta_2} \right) \tau_2 = -\mu c_2,
\end{cases}
\]
The system can be rewritten in matrix form as:
\[
\begin{pmatrix}
	P_1(\lambda^2) + I_1 \xi_n^{\theta_1} & \gamma \beta \xi_n \\
	\gamma \beta \xi_n & P_2(\lambda^2) + I_2 \xi_n^{\theta_2}
\end{pmatrix}
\begin{pmatrix}
	\tau_1 \\
	\tau_2
\end{pmatrix}
=
\begin{pmatrix}
	-\varrho c_1 \\
	-\mu c_2
\end{pmatrix},
\]
where
\[
P_1(\lambda^2) = \varrho \lambda^2 - \alpha \xi_n, \quad P_2(\lambda^2) = \mu \lambda^2 - \beta \xi_n, \quad I_j = \smallint_0^\infty g_j(s) e^{-i\lambda s} ds, \quad j=1,2.
\]
We take the specific choice \( c_1 = 0 \), \( c_2 = -\frac{1}{\mu} \), which reduces the right-hand side to \( \begin{pmatrix} 0 \\ 1 \end{pmatrix} \). By Cramer's rule, we compute the determinants for the solution:
\[
D_2 = \begin{vmatrix}
	P_1(\lambda^2) + I_1 \xi_n^{\theta_1} & 0 \\
	\gamma \beta \xi_n & 1
\end{vmatrix}
= P_1(\lambda^2) + I_1 \xi_n^{\theta_1}.
\]
\[
\begin{aligned}
	D &= \left(P_1(\lambda^2) + I_1 \xi_n^{\theta_1}\right)\left(P_2(\lambda^2) + I_2 \xi_n^{\theta_2}\right) - \gamma^2 \beta^2 \xi_n^2 \\
	&= P_1(\lambda^2)P_2(\lambda^2) - \gamma^2 \beta^2 \xi_n^2 + \underbrace{P_1(\lambda^2) I_2 \xi_n^{\theta_2}}_{J_1} + \underbrace{P_2(\lambda^2) I_1 \xi_n^{\theta_1}}_{J_2} + \underbrace{I_1 I_2 \xi_n^{\theta_1+\theta_2}}_{J_3}.
\end{aligned}
\]
The solution for \( \tau_2 \) is then
\begin{align} \label{eq:4.7}
	\tau_2 = \frac{D_2}{D} = \frac{P_1(\lambda^2) + I_1 \xi_n^{\theta_1}}{P_1(\lambda^2)P_2(\lambda^2) - \gamma^2 \beta^2 \xi_n^2 + J_1 + J_2 + J_3}.
\end{align}

\subsubsection*{Case 1. \(\boldsymbol{\theta_{1} \leq \theta_{2}}\)}
we take
\[
P_1(\lambda^2) P_2(\lambda^2) - \gamma^2 \beta^2 \xi_n^2 + J_1= 0.
\]
Expanding this, we obtain
\begin{align}\label{eq:4.11}
	(\varrho \lambda^2 - \alpha \xi_n) (\mu \lambda^2 - \beta \xi_n)  - \gamma^2 \beta^2 \xi_n^2+ (\varrho \lambda^2 - \alpha \xi_n) I_2 \xi_n^{\theta_2} = 0.
\end{align}
Setting $\lambda^2 = s \xi_n$ and dividing \eqref{eq:4.11} by $\xi_n^2$, we get
\[
F_n(s) = (\varrho s - \alpha) (\mu s - \beta)  - \gamma^2 \beta^2+ (\varrho s - \alpha) I_2 {\frac{1}{\xi_n^{1-\theta_2}}} = 0,
\]
where
\[
I_2 {\frac{1}{\xi_n^{1-\theta_2}}} = O\left( \frac{1}{\xi_n^{\frac{3}{2}-\theta_2}} \right) = o(1).
\]
Let $s_0$ be a positive root of the equation
\[
(\varrho s - \alpha) (\mu s - \beta) - \gamma^2 \beta^2 = 0.
\]
Then
\[
s_0 = \frac{\frac{\beta}{\mu} + \frac{\alpha}{\varrho} + \sqrt{\left( \frac{\beta}{\mu} + \frac{\alpha}{\varrho} \right)^2 - 4 \frac{\alpha_1 \beta}{\varrho \mu}}}{2}.
\]
we have
\[
\frac{\partial}{\partial s} F_n(s) = 2 \varrho \mu s - (\varrho \beta + \mu \alpha) + \frac{\varrho}{\xi_n^{1-\theta_2}} \frac{m_2}{\mu_2 + i \sqrt{s\xi_n} } + (\varrho s - \alpha) \frac{1}{\xi_n^{1-\theta_2}} \frac{-i m_2 \sqrt{\xi_n}}{2 \sqrt{s} (\mu_2 + i \sqrt{s\xi_n} )^2}
\]
and it is continuous for $s \in (s_0 - \delta, s_0 + \delta)$.\\
Therefore, by the Implicit Function Theorem, there exists a sufficiently large $N$ such that for all $n > N$, $F_n(s)$ has a unique solution $s_n$ in a small neighborhood of $s_0$, and $\lim_{n \to \infty} s_n = s_0.$

Introducing the notation $a_n \approx b_n$ when $\lim_{n\to\infty} \frac{|a_n|}{|b_n|}$ is a positive real number.

Hence, $\lambda_n = \sqrt{s_n \xi_n}$, and clearly $\lambda_n \approx \xi_n^{\frac{1}{2}}$. We consider $\lambda := \lambda_n \approx \xi_n^{\frac{1}{2}}$. In this case, equation \eqref{eq:4.7} becomes
\begin{align}\label{eq:4.12}
	\tau_{2,n} = \frac{P_1(\lambda_n^2) + I_1 \xi_n^{\theta_1}}{J_2 + J_3} .
\end{align}
For the exponential kernel, we have
\begin{align} \label{eq:4.9}
	I_j = \smallint_0^\infty m_j e^{-\mu_j s} e^{-i\lambda_n s} ds = \frac{m_j}{\mu_j + i\lambda_n }\approx m_{j}\lambda_{n}^{-1}.
\end{align}
Note that
\[
P_1(\lambda_n^2) = \varrho \lambda_n^2 - \alpha \xi_n = \varrho s_n \xi_n - \alpha \xi_n.
\]
Taking the limit as $n \to \infty$, we get
\[
\begin{aligned}
	\lim_{n \to \infty} P_1(\lambda_n^2) &= \lim_{n \to \infty} (\varrho s_0 \xi_n - \alpha \xi_n) = \lim_{n \to \infty} \frac{\varrho}{2} \left( 2 s_0 \xi_n - \frac{2 \alpha}{\varrho} \xi_n \right) \\
	&= \lim_{n \to \infty} \frac{\varrho}{2} \left( \frac{\beta}{\mu} + \frac{\alpha}{\varrho} + \sqrt{\left( \frac{\beta}{\mu} + \frac{\alpha}{\varrho} \right)^2 - 4 \frac{\alpha_1 \beta}{\varrho \mu}} - \frac{2 \alpha}{\varrho} \right) \xi_n \\
	&= \lim_{n \to \infty} \frac{\varrho}{2} \left( \frac{\beta}{\mu} - \frac{\alpha}{\varrho} + \sqrt{\left( \frac{\beta}{\mu} - \frac{\alpha}{\varrho} \right)^2 + \frac{4 \gamma^2 \beta^2}{\varrho \mu}} \right) \xi_n.
\end{aligned}
\]
Therefore,
\[
P_1(\lambda_n^2) \approx \xi_n \approx \lambda_n^2
\]
Similarly, we have
\[
P_2(\lambda_n^2) \approx \xi_n \approx \lambda_n^2
\]
Moreover,
\begin{align}\label{eq:4.13}
	J_2 \approx m_1 \lambda_n^{1+2\theta_1}, \quad J_3 \approx m_1 m_2 \lambda_n^{-2+2\theta_1+2\theta_2}
\end{align}
From \eqref{eq:4.9} and \eqref{eq:4.13}, we have
\[
\lim_{n \to \infty} \frac{|J_3|}{|J_2|} = 0 \quad \text{and} \quad \lim_{n \to \infty} \frac{|I_1 \xi_n^{\theta_1}|}{|P_1(\lambda_n^2)|} = 0
\]
Thus,
\[
\tau_{2,n} \approx \frac{P_1(\lambda_n^2)}{J_2} \approx \lambda_n^{1-2\theta_1}
\]
in this point, if \(X_n=(v_n,u_n,p_n,q_n,\eta_n,\theta_n)\) is the solution of the system \((i\lambda_n I-B)X_n = F_n\), then we obtain
\[
\norm{X_n} \geq \mu^{\frac{1}{2}} \norm{q_n} = \mu^{\frac{1}{2}} \norm{i\lambda_n p_n} = \mu^{\frac{1}{2}} |\lambda_n| |\tau_{2,n}| \approx \mu^{\frac{1}{2}} \lambda_n^{2 - 2\theta_1}\geq \varepsilon_2 \lambda_n^{2 - 2\theta_1}
\]
Furthermore, if the semigroup of the system decays polynomially with the rate \(t^{-\delta }\) with  \(\delta \ge \frac{1}{2-2\theta_{1}} \) , we have
\[
\varepsilon_2 \lambda_n^{2 - 2\theta_1} \leq \norm{X_n} \leq C \lambda_n^{1/\delta} \norm{F_n} \implies \varepsilon_2 \lambda_n^{2 - 2\theta_1-1/\delta}  \leq C .
\]
which is contradictory because  $\lambda_n^{2-2\theta_1-1/\sigma} \to \infty$ when $n \to \infty$. Thus, the decay rate $t^{-\frac{1}{2-2\theta_{1}}}$ is optimal.

\subsubsection*{Case 2. $\boldsymbol{\theta_2 \leq \theta_1}$}
Choosing $c_1 = -\frac{1}{\varrho}$ and $c_2 = 0$, the right-hand side becomes\( \begin{pmatrix} 1 \\ 0 \end{pmatrix} \), we can solve for $\tau_1$
\begin{align}\label{eq:4.14}
	\tau_1 = \frac{P_2(\lambda^2) + I_2 \xi_n^{\theta_2}}{P_1(\lambda^2) P_2(\lambda^2) - \gamma^2 \beta^2 \xi_n^2 + J_1 + J_2 + J_3}
\end{align}
Now taking
\[
P_1(\lambda^2) P_2(\lambda^2) - \gamma^2 \beta^2 \xi_n^2 + J_2 = 0.
\]
By the same method, we obtain $\lambda = \lambda_n = \sqrt{s_n\xi_n} $ with $\lim_{n \to \infty} s_n = s_0$. Then equation \eqref{eq:4.14} becomes
\begin{align}\label{eq:4.15}
	\tau_{1,n} = \frac{P_2(\lambda_n^2) + I_2 \xi_n^{\theta_2}}{J_1 + J_3}.
\end{align}
Similarly, we have
\[
P_1(\lambda_n^2) \approx \xi_n \approx \lambda_n^2, \quad P_2(\lambda_n^2) \approx \xi_n \approx \lambda_n^2.
\]
Therefore,
\begin{align}\label{eq:4.16}
	J_1 \approx m_2 \lambda_n^{1+2\theta_2}, \quad J_3 \approx m_1 m_2 \lambda_n^{-2+2\theta_1+2\theta_2}.
\end{align}
From \eqref{eq:4.9}, \eqref{eq:4.16}, we get
\[
\lim_{n \to \infty} \frac{|J_3|}{|J_1|} = 0, \quad \lim_{n \to \infty} \frac{|I_2 \xi_n^{\theta_2}|}{|P_2(\lambda_n^2)|} = 0.
\]
Then, we obtain
\[
\tau_{1,n} \approx \frac{P_2(\lambda_n^2)}{J_1} \approx \lambda_n^{1-2\theta_2}
\]
Furthermore, if the semigroup of the system decays polynomially with the rate \(t^{-\delta }\) with  \(\delta \ge \frac{1}{2-2\theta_{2}} \) , we obtain
\[
\norm{X_n} \geq \mu^{\frac{1}{2}} \norm{q_n} = \mu^{\frac{1}{2}} \norm{i\lambda_n p_n} = \mu^{\frac{1}{2}} |\lambda_n| |\tau_{2,n}| \approx \mu^{\frac{1}{2}} \lambda_n^{2 - 2\theta_2}\geq \varepsilon_3 \lambda_n^{2 - 2\theta_2}
\]

In this moment, if the semigroup of the system decays polynomially with the rate \(t^{-\delta }\) with  \(\delta \ge \frac{1}{2-2\theta_{2}} \) , we have
\[
\varepsilon_3 \lambda_n^{2 - 2\theta_2} \leq \norm{X_n} \leq C \lambda_n^{1/\delta} \norm{F_n} \implies \varepsilon_2 \lambda_n^{2 - 2\theta_2-1/\delta}  \leq C .
\]
which is contradictory since $\lambda_n^{2-2\theta_2-1/\sigma} \to \infty$ when $n \to \infty$. In this case, the optimal decay rate is $t^{-\frac{1}{2-2\theta_{2}}}$. In summary, the optimal decay rate is $t^{-\frac{1}{2-2\theta_{0}}}$,  $\theta_{0} = \min\{\theta_1,\theta_2\}$. The proof is completed.
\end{proof}

	\subsection{Exponential stability with $\theta_1=\theta_2=1$}
	The main result established in this section is stated in the following theorem.
\begin{theorem}\label{thmex}
	Suppose that assumptions (A1)--(A3) are satisfied. If $\theta_1=\theta_2=1$, then the semigroup $S(t)$ generated by the problem \eqref{modelmain} is  exponentially stable. Specifically, for any initial state $X_0 \in \mathcal{H}$, the estimate
		\begin{align*}
\|S(t)X_0\|_{\mathcal{H}} \leq C e^{-rt} \|X_0\|_{\mathcal{H}}, \quad \forall t \geq 0,
		\end{align*}
holds for some positive constant s $C $ and $r$.
	\end{theorem}
\begin{proof}
By Theorem \ref{thmG}, we only need to show $i\R\subset\varrho(B)$ and \eqref{eq:lemma2} holds. Since $i\R\subset\varrho(B)$ follows exactly the same lines as Lemma \ref{lemthm1-1}. We only need to prove \eqref{eq:lemma2}. By contradiction, assume \eqref{eq:lemma2} is not true. Then there exists a sequence $X_n = (v_n, u_n, p_n, q_n, \eta_n, \theta_n) \in \D(B)$,
		with $\|X_n\|_{\mathcal{H}} = 1$ and $\lambda_n \to \infty$, such that
		\[
		(i\lambda_n - B)X_n = (f_n^1, f_n^2, f_n^3, f_n^4, f_n^5, f_n^6) \to 0 \quad \text{in } \mathcal{H}.
		\]
	Expanding the equation, we obtain,
		\begin{align}
			&i\lambda_{n}v_{n} - u_{n} = f_{n}^{1} \to 0 \quad \text{in } D(A^{\frac{1}{2}}),\label{eq:Th3.3.1}\\
			&i\lambda_{n}u_{n} - \frac{1}{\varrho}\Bigl(-\alpha Av_{n} + \gamma\beta Ap_{n}
			+ \kapp_{1}Av_{n} - \smallint_{0}^{\infty}g_{1}(s)A\eta_{n}\,ds\Bigr)
			= f_{n}^{2} \to 0\quad \text{in } H,  \label{eq:Th3.3.2}\\
			&i\lambda_{n}p_{n} - q_{n} = f_{n}^{3} \to 0\quad \text{in } D(A^{\frac{1}{2}}), \label{eq:Th3.3.3}\\
			&i\lambda_{n}q_{n} - \frac{1}{\mu}\Bigl(-\beta Ap_{n} + \gamma\beta Av_{n}
			+ \kapp_{2}Ap_{n} - \smallint_{0}^{\infty}g_{2}(s)A\theta_{n}\,ds\Bigr)
			= f_{n}^{4} \to 0\quad \text{in } H, \label{eq:Th3.3.4}\\
			&i\lambda_{n}\eta_{n} + \eta_{n,s} - u_{n} = f_{n}^{5} \to 0\quad \text{in } \M{1}, \label{eq:Th3.3.5}\\
			&i\lambda_{n}\theta_{n} + \theta_{n,s} - q_{n} = f_{n}^{6} \to 0 \quad \text{in } \M{2}. \label{eq:Th3.3.6}
		\end{align}

By the same proof as \eqref{eq:Th3.2.1c} and \eqref{eq:Th3.2.2c}, it follows
		\begin{align}
		\|\eta_n\|_{\M{1}} =o(1), \label{eq:Th3.3.7}\\
		\|\theta_n\|_{\M{2}}=o(1). \label{eq:Th3.3.8}
		\end{align}
		
		Taking the $H$-inner product of equation \eqref{eq:Th3.3.5} with $g_1(s) u_n$ and integrating from $0$ to $\infty$ with respect to $s$, we get
		\[
		\smallint_0^\infty g_1(s) \left\langle i\lambda_n \eta_n, u_n \right\rangle ds + \smallint_0^\infty g_1(s) \left\langle \eta_{n,s}, u_n \right\rangle ds - \smallint_0^\infty g_1(s) \left\langle u_n, u_n \right\rangle ds =\smallint_0^\infty g_1(s)\left\langle f^5_n, u_n \right\rangle ds.
		\]
		For the first term, by equation \eqref{eq:Th3.3.2}, using the Cauchy-Schwarz inequality, the continuous embedding $D(A^{\frac{1}{2}}) \hookrightarrow H$ and \eqref{eq:Th3.3.7}, we have
		\[
		\begin{aligned}
			&\left| \smallint_0^\infty g_1(s) \left\langle i\lambda_n \eta_n, u_n \right\rangle ds \right| = \left| \smallint_0^\infty g_1(s) \left\langle \eta_n, i\lambda_n u_n \right\rangle ds \right| \\
			&\leq \frac{1}{\varrho} \left| \smallint_0^\infty g_1(s) \left\langle \eta_n, -\alpha A v_n \right\rangle ds + \smallint_0^\infty g_1(s) \left\langle \eta_n, \gamma \beta A p_n \right\rangle ds \right. \\
			&\quad \left. + \smallint_0^\infty g_1(s) \left\langle \eta_n, \kappa_1 A v_n \right\rangle ds - \smallint_0^\infty g_1(s) \left\langle \eta_n, \smallint_0^\infty g_1(s) A \eta_n ds \right\rangle ds \right| + \left| \smallint_0^\infty g_1(s) \left\langle \eta_n , f_n^2 \right\rangle \right|\\
			&\leq \frac{1}{\varrho} \left\{ \alpha \left\| A^{\frac{1}{2}} v_n \right\| \sqrt{\kappa_1} \|\eta_n\|_{\M{1}} + \gamma \beta \left\| A^{\frac{1}{2}} p_n \right\| \sqrt{\kappa_1} \|\eta_n\|_{\M{1}} \right. \\
			&\quad \left. + \kappa_1 \left\| A^{\frac12} v_n \right\| \sqrt{\kappa_1} \|\eta_n\|_{\M{1}} + \kappa_1 \|\eta_n\|_{\M{1}}^2 \right\} + \sqrt{\kappa_1} \|f_n^2\| \|\eta_n\|_{\M{1}}^2 =o(1),
		\end{aligned}
		\]
		where we used the boundedness of $\| A^{\frac{1}{2}} v_n \|$,  $\| A^{\frac{1}{2}} p_n \|$, $\|f_n^2\|$ and \eqref{eq:Th3.3.7}.

		For the second term, using integration by parts, the Cauchy-Schwarz inequality, the continuous embedding $D(A^{\frac{1}{2}}) \hookrightarrow H$ with embedding constant $C$, \eqref{eq:Th3.3.7} and the assumption (A2), we get
		\[
		\begin{aligned}
			\left| \smallint_0^\infty g_1(s) \left\langle \eta_{n,s}, u_n \right\rangle ds \right| &= \left| -\smallint_0^\infty g_1'(s) \left\langle \eta_n, u_n \right\rangle ds \right| \leq c_0 \left| \smallint_0^\infty g_1(s) \left\langle \eta_n, u_n \right\rangle ds \right| \\
			&\leq c_0 \|u_n\| \smallint_0^\infty g_1(s) \|\eta_n\| ds \leq c_0C \|u_n\| \smallint_0^\infty g_1(s) \left\| A^{\frac{1}{2}} \eta_n \right\| ds \\
			&\leq c_0C \|u_n\| \sqrt{\kappa_1} \|\eta_n\|_{\M{1}} =o(1),
		\end{aligned}
		\]
		where we used the boundedness of $\|u_n \|$ and \eqref{eq:Th3.3.7}.

		For the right-hand side, using Cauchy-Schwarz inequality, the continuous embedding $D(A^{\frac{1}{2}}) \hookrightarrow H$, the boundedness of $\|u_n\|$ and $\|f^5_n\|_{\M{1}} = o(1)$, we obtain
		\begin{align*}
		\left|\smallint_0^\infty g_1(s)\left\langle f^5_n, u_n \right\rangle ds\right| \leq \smallint_0^\infty g_1(s)\| f^5_n\| \| u_n\| ds \leq C\smallint_0^\infty g_1(s)\| A^{\frac{1}{2}}f^5_n\| \| u_n\| ds \leq C\|u_n\|\sqrt{\kappa_1} \|f^5_n\|_{\M{1}} =o(1).
	    \end{align*}
Then above analysis shows that
		\[
		\left| \smallint_0^\infty g_1(s) \left\langle u_n, u_n \right\rangle ds \right| = \kappa_1 \|u_n\|^2 =o(1),
		\]
		which implies
		\begin{align}
			\|u_n\| =o(1). \label{eq:Th3.3.9}
		\end{align}
		Similarly, we can obtain
		\begin{align}
			\|q_n\| =o(1).  \label{eq:Th3.3.10}
		\end{align}
		By the same argument as in the proof of Lemma 3.4, it still hold that
		\begin{align}\label{eq:Th3.3.11}
		\alpha_1 \left\| A^{\frac{1}{2}} v_n \right\|^2 - \kappa_1 \left\| A^{\frac{\theta_1}{2}} v_n \right\|^2 + \beta \left\| \gamma A^{\frac{1}{2}} v_n - A^{\frac{1}{2}} p_n \right\|^2 - \kappa_2 \left\| A^{\frac{\theta_2}{2}} p_n \right\|^2 =o(1).
		\end{align}
		Combining \eqref{eq:Th3.3.7}, \eqref{eq:Th3.3.8}, \eqref{eq:Th3.3.9}, \eqref{eq:Th3.3.10} and \eqref{eq:Th3.3.11}, we get $\|X_n\|_{\mathcal{H}} =o(1)$, which contradicts $\|X_n\|_{\mathcal{H}} = 1$, hence Theorem \ref{thmex} holds. The proof is completed.
\end{proof}


\end{document}